\documentclass[reqno]{amsart}
\usepackage{amsthm,amsmath,amssymb}
\usepackage{stmaryrd,mathrsfs}

\newcommand{\ud}[0]{\,\mathrm{d}}

\newcommand{\ceil}[1]{\lceil #1 \rceil}
\newcommand{\Bceil}[1]{\Big\lceil #1 \Big\rceil}
\newcommand{\dist}[0]{\operatorname{dist}}

\newcommand{\abs}[1]{|#1|}
\newcommand{\Babs}[1]{\Big|#1\Big|}
\newcommand{\norm}[2]{|#1|_{#2}}
\newcommand{\Bnorm}[2]{\Big|#1\Big|_{#2}}
\newcommand{\Norm}[2]{\|#1\|_{#2}}
\newcommand{\bNorm}[2]{\big\|#1\big\|_{#2}}
\newcommand{\BNorm}[2]{\Big\|#1\Big\|_{#2}}
\newcommand{\pair}[2]{\langle #1,#2 \rangle}
\newcommand{\Bpair}[2]{\Big\langle #1,#2 \Big\rangle}
\newcommand{\ave}[1]{\langle #1\rangle}

\newcommand{\bddlin}[0]{\mathscr{L}}

\newcommand{\id}[0]{\operatorname{id}}
\newcommand{\BMO}[0]{\operatorname{BMO}}
\newcommand{\supp}[0]{\operatorname{supp}}
\newcommand{\Car}[0]{\operatorname{Car}}
\newcommand{\loc}[0]{\operatorname{loc}}

\newcommand{\R}{\mathbb{R}}
\newcommand{\C}{\mathbb{C}}
\newcommand{\N}{\mathbb{N}}
\newcommand{\Z}{\mathbb{Z}}

\newcommand{\prob}[0]{\mathbb{P}}
\newcommand{\Exp}[0]{\mathbb{E}}
\newcommand{\D}[0]{\mathbb{D}}
\newcommand{\radem}[0]{\varepsilon}
\newcommand{\rbound}[0]{\mathscr{R}}

\newcommand{\good}[0]{\operatorname{good}}
\newcommand{\bad}[0]{\operatorname{bad}}

\newcommand{\sep}[0]{\operatorname{sep}}

\newcommand{\comment}[1]{}
\newcommand{\ontop}[2]{\begin{smallmatrix} #1 \\ #2 \end{smallmatrix}}

\numberwithin{equation}{section}

\makeatletter
  \let\c@equation\c@subsection
\makeatother

\theoremstyle{plain}
\newtheorem{tbthm}{\textit{Tb} theorem}

\newtheorem{theorem}[subsection]{Theorem}
\newtheorem{proposition}[subsection]{Proposition}
\newtheorem{corollary}[subsection]{Corollary}
\newtheorem{lemma}[subsection]{Lemma}

\theoremstyle{definition}

\theoremstyle{remark}
\newtheorem{remark}[subsection]{Remark}

\newtheorem{notation}[subsection]{Notation}

\begin{document}

\title{The vector-valued non-homogeneous \(Tb\) theorem}

\author[T.~P.\ Hyt\"onen]{Tuomas P.\ Hyt\"onen}
\thanks{The author was supported by the Academy of Finland through the projects 114374 ``Vector-valued singular integrals'' and 130166 ``$L^p$ methods in harmonic analysis.''}
\address{Department of Mathematics and Statistics, University of Helsinki, P.O.B. 68, FI-00014 Helsinki, Finland}
\email{tuomas.hytonen@helsinki.fi}

\date{\today}



\begin{abstract}
The paper gives a Banach space -valued extension of the $Tb$ theorem of Nazarov, Treil and Volberg (2003) concerning the boundedness of singular integral operators with respect to a measure $\mu$, which only satisfies an upper control on the size of balls. Under the same assumptions as in their result, such operators are shown to be bounded on the Bochner spaces $L^p(\mu;X)$ of functions with values in $X$---a Banach space with the unconditionality property of martingale differences (UMD). The new proof deals directly with all $p\in(1,\infty)$ and relies on delicate estimates for the non-homogenous ``Haar'' functions, as well as McConnell's (1989) decoupling inequality for tangent martingale differences.
\end{abstract}

\maketitle

\tableofcontents

\section{Introduction}\label{intro}

The aim of this paper is to bring together two so-far distinct lines along which the classical Calder\'on--Zygmund theory has been generalized: one of them related to the domain, the other to the range of the functions under consideration. On the one hand, there has been considerable interest in singular integrals with respect to quite general measures (in particular, ones failing the doubling hypothesis), and a fairly complete theory is now available especially due to the efforts of Nazarov, Treil and Volberg \cite{NTV:Cauchy,NTV:weak,NTV:system,NTV:Tb}, and Tolsa~\cite{Tolsa:Cauchy,Tolsa:RBMO,Tolsa:LP}.

In another direction, where pioneering contributions were made by Bourgain~\cite{Bourgain:83,Bourgain:86} and Burkholder~\cite{Burkholder}, much of the classical theory of singular integrals has been extended to the setting of functions which take their values in an infinite-dimensional Banach space.
By the end of the 1980's, this theory had already advanced up to the vector-valued \(T1\) theorem proved by Figiel~\cite{Figiel:90}. A more recent twist to this second line, boosted by the work of Weis~\cite{Weis}, is the further generalization to operator-valued integral kernels, 
although still in the homogeneous (and in most cases, Euclidean--Lebesguean) situation as far as the underlying measure space is concerned.

It seems natural to ask for a unification: a vector-valued, non-homogeneous Calder\'on--Zygmund theory which would be a common generalization of the two lines of development described above. In fact, the methods of proof in the two fields are already quite suggestive of such a convergence, the interplay of probability and analysis being in the centre: Ever since the pioneering contributions, the vector-valued theory has heavily relied on probabilistic tools, especially martingale differences and their unconditionality (UMD), which is the defining property of the class of admissible spaces for most results. Also in the non-homogeneous \(Tb\) theorem~\cite{NTV:Tb}, martingale differences were employed to construct the basic decomposition of the operator, and Nazarov, Treil and Volberg have added further probabilistic ingredients which are decisive for their analysis.

\medskip

I now recall the hypotheses of the \(Tb\) theorem of Nazarov et al.\ concerning the underlying measure space and the associated Calder\'on--Zygmund operators; this will also be basic set-up of the present paper.
Let \(\mu\) be a Borel measure on \(\R^N\) which satisfies, for a real number \(d\in(0,N]\), the upper bound
\begin{equation*}
  \mu(B(x,r))\leq r^d
\end{equation*}
for any ball \(B(x,r)\) of centre \(x\in\R^N\) and radius \(r>0\). A \emph{\(d\)-dimensional Calder\'on--Zygmund kernel} is a function \(K(x,y)\) of variables \(x,y\in\R^N\), \(x\neq y\), which satisfies
\begin{equation}\label{eq:CZK1}
  \abs{K(x,y)}\leq \frac{1}{\abs{x-y}^{d}},
\end{equation}
\begin{equation}\label{eq:CZK2}
  \abs{K(x,y)-K(x',y)}+\abs{K(y,x)-K(y,x')}\leq \frac{\abs{x-x'}^{\alpha}}{\abs{x-y}^{d+\alpha}}
\end{equation}
for some \(\alpha>0\) and all variables such that $\abs{x-y}>2\abs{x-x'}$. Of course one could allow multiplicative constants in these assumptions (and some which follow), but since there will be quite many parameters involved in any case and the full generality is reached by trivial scaling arguments, I will restrict myself to the normalized situation above.

Let \(T:f\mapsto Tf\) be a linear operator acting on some functions \(f\) (this will be specified in more detail shortly). It is called a \emph{Calder\'on--Zygmund operator} with kernel \(K\) if
\begin{equation}\label{eq:Tf}
  Tf(x)=\int_{\R^N}K(x,y)f(y)\ud\mu(y)
\end{equation}
for \(x\) outside the support of \(f\).

An operator \(T\) is said to satisfy the \emph{rectangular weak boundedness property} if for all rectangles \(R\) there holds
\begin{equation*}
  \Babs{\int_{\R^N} 1_R\cdot T1_R \ud\mu}\leq\mu(R);
\end{equation*}
as usual in the related literature, a rectangle here means a set of the form \(R=x_0+\prod_{i=1}^N[-\ell_i/2,\ell_i/2)\subset\R^N\). The special case with \(\ell_i=\ell\) for all \(i\) is called a cube, and in this case \(\ell(R):=\ell\) designates its side-length. For a cube \(Q\) and \(\lambda>0\), \(\lambda Q\) is the unique cube with the same centre and \(\lambda\) times the radius of \(Q\).

A function \(b\in L^1_{\loc}(\mu)\) is called \emph{weakly accretive} if
\begin{equation*}
  \frac{1}{\mu(Q)}\Babs{\int_Q b\ud\mu}\geq\delta
\end{equation*}
for all cubes \(Q\) and some fixed \(\delta>0\). I fix two weakly accretive functions \(b_1\) and \(b_2\), which satisfy the above estimate and in addition \(\Norm{b_i}{\infty}\leq 1\). Below, the weak boundedness property will be assumed for the composition of operators \(M_{b_2}TM_{b_1}\), where \(M_b:f\mapsto b\cdot f\) designates the operator of pointwise multiplication by \(b\).

A funtion \(h\in L^1_{\loc}(\mu)\) is said to be in \(\BMO_{\lambda}^p(\mu)\), where \(\lambda,p\in[1,\infty)\), if
\begin{equation}\label{eq:defBMO}
  \Norm{h}{\BMO_{\lambda}^p(\mu)}:=\sup_Q\Big(\frac{1}{\mu(\lambda Q)}\int_Q\abs{h-\ave{h}_Q}^p\ud\mu\Big)^{1/p}<\infty,
\end{equation} 
where the supremum is over all cubes \(Q\subset\R^N\). Here \(\ave{h}_Q:=\mu(Q)^{-1}\int_Q h\ud\mu\) is the average of \(h\) on \(Q\). Let some \(\lambda>1\) be fixed from now on.

Let then \(X\) be a Banach space and \(L^p(\mu;X)\) designate the Bochner space of \(\mu\)-measurable \(X\)-valued functions with its usual norm. The question of interest in this paper is the boundedness of \(T\) on \(L^p(\mu;X)\).
For the sake of simplicity, I will concentrate on the quantitative aspect of this problem: I will assume that \(T\) is in fact defined as a continuous linear operator on the whole space \(L^p(\mu;X)\) from the beginning, but I then derive a bound \(C\) for its operator norm according to the following convention:

\begin{notation}\label{not:C}
The letter \(C\) will always indicate a finite quantity, which depends at most on the following set of parameters:
\begin{equation*}
  d,N,p,X,\alpha,\delta,\lambda, 
\end{equation*}
plus a few auxiliary ones which will be explicitly introduced below and eventually chosen in such a way that they, too, only depend on the above-mentioned list.
The numerical value of \(C\) need not be the same from one occurrence to another. An estimate of the type \(F\leq CG\) will sometimes be abbreviated to \(F\lesssim G\), and \(F\lesssim G\lesssim F\) to \(F\eqsim G\).
\end{notation}

Various ways of reducing to the a priori bounded situation have been discussed by Nazarov et al.~\cite{NTV:Tb}; here I point out just one more strategy, which is specific to the present vector-valued context: One starts by considering \(T\) on functions taking values in a finite-dimensional subspace \(X_0\subset X\). By choosing a basis of $X_0$ and considering the action of $T$ componentwise, it easily follows from the boundedness of \(T\) on \(L^p(\mu)\) (which is the conclusion of the scalar-valued \(Tb\) theorem) that it is also bounded on \(L^p(\mu;X_0)\), but the bounds resulting from such a simple argument will grow as a function of \(\dim X_0\). However, once it is shown that the norm of \(T\) on \(L^p(\mu;X_0)\) is actually bounded by a constant \(C\) independent of \(X_0\subset X\), it also follows that \(T\) extends continuously to all of \(L^p(\mu;X)\) by the density of functions with a finite-dimensional range.

It is well known that the typical singular integral operators \(T\) will not extend boundedly to \(L^p(\mu;X)\) for an arbitrary Banach space \(X\). In fact, the classical Hilbert transform \(H\) satisfies \(H\in\bddlin(L^p(\R;X))\) if (Burkholder~\cite{Burkholder}) and only if (Bourgain~\cite{Bourgain:83}) \(X\) has the UMD property, i.e., there holds
\begin{equation}\label{eq:UMD}
  \BNorm{\sum_{k=1}^n\epsilon_k d_k}{L^p(\mu;X)}
  \leq C\BNorm{\sum_{k=1}^n d_k}{L^p(\mu;X)}
\end{equation}
whenever \((d_k)_{k=1}^n\) is a martingale difference sequence in \(L^p(\mu;X)\), and \(\epsilon_k=\pm 1\). This property is known to be independent of the parameter \(p\in(1,\infty)\), and also its validity for dyadic martingales with respect to the Lebesgue measure already implies the general condition (Maurey~\cite{Maurey}). UMD implies reflexivity but not conversely, although all the ``usual'' reflexive spaces (such as the reflexive Lebesgue, Sobolev, and Besov spaces, and also the noncommutative \(L^p\) spaces) do have UMD.

It is now possible to formulate the main result:

\begin{tbthm}\label{thm:main}
Let \(X\) be a UMD space and \(1<p<\infty\).
Let \(T\) be a Calder\'on--Zygmund operator for which \(M_{b_2}TM_{b_1}\) satisfies the rectangular weak boundedness property and 
\begin{equation}\label{eq:TbBMO}
  \Norm{Tb_1}{\BMO_{\lambda}^1(\mu)}\leq 1,\qquad
  \Norm{T^*b_2}{\BMO_{\lambda}^1(\mu)}\leq 1.
\end{equation}
Then \(\Norm{T}{\bddlin(L^p(\mu;X))}\leq C\).
\end{tbthm}

The case \(X=\C\) is a version of the celebrated \(Tb\) theorem of Nazarov, Treil and Volberg~\cite{NTV:Tb}. Its known proof consists of two methodically distinct and essentially decoupled main parts, as in the classical Calder\'on--Zygmund theory. First, the \(L^2\) estimate \(\Norm{T}{\bddlin(L^2(\mu))}\leq  C\) is proved by exploiting, of course, the Hilbert space structure of \(L^2(\mu)\). Second---although historically this step preceded the first one, and was proven in the non-homogeneous context by Nazarov, Treil and Volberg in~\cite{NTV:weak}---, some weak-type \(L^1\) estimates are deduced, and here one employs the kernel conditions \eqref{eq:CZK1} and \eqref{eq:CZK2} plus the already established (or historically, postulated) \(L^2\) bound. The inequality \(\Norm{T}{\bddlin(L^p(\mu))}\leq C\) for \(p\in(1,2)\cup(2,\infty)\) then follows from the abstract principles of interpolation and duality, so it is in this sense reached somewhat indirectly.
The present contribution, as a byproduct of the vector-valued extension, also offers a new approach to the scalar-valued result in \(L^p(\mu)\), which is more direct than the one just outlined for \(p\neq 2\).

Of course, the rectangular weak boundedness property and the \(\BMO\) conditions \eqref{eq:TbBMO} are also necessary for \(Tb\) theorem~\ref{thm:main}, since they are necessary in the scalar-valued case, and one can identify \(L^p(\mu)\) as a subspace of \(L^p(\mu;X)\) by considering functions with values in any one-dimensional subspace of \(X\). One could also allow only the more restricted cubic weak boundedness property with parameter \(\Lambda\geq 1\):
\begin{equation*}
  \Babs{\int_{\R^N}1_Q\cdot T1_Q\ud\mu}\leq\mu(\Lambda Q)
\end{equation*}
for all cubes \(Q\subset\R^N\). The vector-valued proof could be extended to this situation, but the somewhat tedious refinements needed in the argument would be more or less a repetition of the corresponding steps from \cite{NTV:Tb}. Instead, this extension can be easily deduced from the work already done in the scalar case:

\begin{tbthm}\label{thm:TbWBP}
Assume the conditions of \(Tb\) theorem~\ref{thm:main}, except that the rectangular weak boundedness property of \(M_{b_2} T M_{b_1}\) is replaced by the cubic weak boundedness property with parameter \(\Lambda\geq 1\). Then \(\Norm{T}{\bddlin(L^p(\mu;X))}\leq C\), where \(C\) is also allowed to depend on \(\Lambda\).
\end{tbthm}

\begin{proof}
By Nazarov, Treil and Volberg's \(Tb\) theorem, \(\Norm{T}{\bddlin(L^2(\mu))}\leq C\), hence \(M_{b_2}TM_{b_1}\) satisfies the rectangular weak boundedness property. Thus Theorem~\ref{thm:main} applies.
\end{proof}

The necessity of the assumptions may also be exploited to derive the following immediate but interesting variant:

\begin{tbthm}\label{thm:TbL2}
Let \(T\) be a Calder\'on--Zygmund operator with \(\Norm{T}{\bddlin(L^2(\mu))}\leq 1\).
Let \(X\) be a UMD space and \(1<p<\infty\). Then \(\Norm{T}{\bddlin(L^p(\mu;X))}\leq C\).
\end{tbthm}

\begin{proof}
By the converse part of the \(Tb\) (or just \(T1\)) theorem of Nazarov, Treil and Volberg, \(T\) satisfies the rectangular weak boundedness property and \(T1,T^*1\in\BMO_{\lambda}^1(\mu)\). Hence \(Tb\) theorem~\ref{thm:main} (with \(b_1=b_2= 1\)) applies.
\end{proof}

This allows, e.g., to use the conditions of the accretive system \(Tb\) theorem of Nazarov, Treil and Volberg~\cite{NTV:system} (which, by their result, imply the \(L^2(\mu)\)-boundedness) for checking the \(L^p(\mu;X)\)-boundedness of a Calder\'on--Zygmund operator.

\medskip

In the spirit of the recent vector-valued results \cite{H:Tb,HW:T1}, \(Tb\) theorem~\ref{thm:main} also admits a generalization in the context of operator-valued kernels. Integral transformations with such kernels arise for instance when solving abstract differential equations in a Banach space, where much of the motivation for this kind of considerations originally came from; see Weis~\cite{Weis}. From Weis' work and the subsequent developments, it has been known for some time that for boundedness results analogous to the scalar-kernel case to be valid, one needs to impose conditions which are stronger than the first guess ``replace all absolute values by norms.'' Recall that an operator family \(\mathscr{T}\subset\bddlin(X)\) is called \emph{Rademacher-bounded}, or \(R\)-bounded, if there is a constant \(c\) such that for all \(n\in\Z_+\), all \(\xi_1,\ldots,\xi_n\in X\) and \(T_1,\ldots,T_n\in \mathscr{T}\),
\begin{equation}\label{eq:defRbd}
  \BNorm{\sum_{k=1}^n\radem_k T_k \xi_k}{L^2(\Omega;X)}
  \leq c\BNorm{\sum_{k=1}^n\radem_k \xi_k}{L^2(\Omega;X)},
\end{equation}
where \(\radem_k\) are the Rademacher functions, as above. Denote the smallest admissible \(c\) by \(\rbound(\mathscr{T})\) and recall the fundamental contraction principle (\cite{DJT}, 12.2), which in this language says that \(\rbound(\Lambda\cdot\id_X)\leq 2\sup_{\lambda\in\Lambda}\abs{\lambda}\) for \(\Lambda\subset\C\); this is the most important tool in handling random series as above, which will be present throughout the proofs of the various \(Tb\) theorems here.

The rule of thumb, which has guided the recent progress with operator-valued kernels, is to replace the boundedness assumptions for scalar kernels by the corresponding Rademacher-boundedness statements in the operator-valued case.
The following operator-valued \(Tb\) theorem implements this idea in the present situation. 
I give a concise statement here, and refer the reader to Section~\ref{sec:operator} for a detailed explanation of the assumptions.

\begin{tbthm}\label{thm:operator}
Let \(X\) be a UMD space and \(1<p<\infty\). 
Let \(T\) be an \(\bddlin(X)\)-valued Rademacher--Calder\'on--Zygmund operator for which \(M_{b_2}TM_{b_1}\) satisfies the rectangular weak Rademacher boundedness property. Let \(Y\subset\bddlin(X)\) and \(Z\subset\bddlin(X^*)\) be subspaces with the UMD property, and
\begin{equation*}
  \Norm{Tb_1}{\BMO_{\lambda}^{p}(\mu;Y)}\leq 1,\qquad
  \Norm{T^*b_2}{\BMO_{\lambda}^{p'}(\mu;Z)}\leq 1.
\end{equation*}
Then \(\Norm{T}{\bddlin(L^p(\mu;X))}\leq C\), where \(C\) is allowed to depend on \(Y\) and \(Z\), in addition to the usual parameters.
\end{tbthm}

For all practical purposes, \(Tb\) theorem~\ref{thm:operator} is a generalization of my operator-valued \(Tb\) theorem for the Lebesgue measure~\cite{H:Tb}, although there are minor technical points (slightly different notions of para-accretivity and weak boundedness, and the treatment in~\cite{H:Tb} of kernels \(K\) satisfying just a logarithmic version of the H\"older continuity in~\eqref{eq:CZK2}) which prevent the above result from strictly covering the earlier one. On the other hand, even for the Lebesgue measure, \(Tb\) theorem~\ref{thm:operator} improves on that of~\cite{H:Tb} in one important respect: the subspaces \(Y\) and \(Z\) are now only required to have UMD; the additional condition imposed in~\cite{H:Tb}, that their unit balls be Rademacher-bounded subsets of \(\bddlin(X)\) and \(\bddlin(X^*)\), is seen to be superfluous by the new techniques.

\begin{remark}
Since the circulation of the first preprint of this paper, the manuscript has undergone quite a substantial evolution. As some citations to this paper have already been made, based on what was written in the earlier versions, it seems appropriate to comment a little on these developments.

In my original formulation of all the $Tb$ theorems above, I also needed to impose another restriction on the Banach space $X$, in addition to the necessary UMD condition. This was the so-called \emph{RMF property}, recently introduced by McIntosh, Portal and myself \cite{HMP}, which means the boundedness $M_R:L^p(\mu;X)\to L^p(\mu)$ of the \emph{Rademacher maximal function}
\begin{equation*}
  M_Rf(x):=\rbound\Big(\Big\{\frac{1}{\mu(Q)}\int_Q f\ud\mu;Q\owns x\Big\}\Big),
\end{equation*}
where the $R$-bound is over all dyadic cubes $Q$ containing $x\in\R^N$, and the Banach space $X$ is identified with the operator space $\bddlin(\C;X)$ (or $\bddlin(\R;X)$) in a canonical way for the computation of the $R$-bound. This condition was originally studied \cite{HMP} only in the case when $\ud\mu=\ud x$, but Kemppainen~\cite{Kemppainen} has shown, in analogy to Maurey's classical result for UMD \cite{Maurey}, that the RMF property with respect to the Lebesgue measure already implies it for other measures as well, and it is also independent of the parameter $p\in(1,\infty)$ appearing in its definition \cite{HMP,Kemppainen}. This notion and the related results played an important r\^ole in finding the original weaker versions of the $Tb$ theorems above.

The RMF property was related to the estimation of the \emph{paraproduct} parts of the operator $T$ and could originally be avoided (trivially) if $Tb_1=T^*b_2=0$ and also (non-trivially but straightforwardly) in the case  when the measure $\mu$ satisfies the doubling condition $\mu(B(x,2r))\leq K\mu(B(x,r))$. Later on \cite{H:Eichstaett}, I also found a somewhat complicated argument to eliminate the RMF assumption under the  condition that $Tb_1,T^*b_2\in L^{\infty}(\mu)$, and conjectured that it should be eliminated altogether. The related ideas, although not exactly along the lines suggested in \cite{H:Eichstaett}, eventually led to the resolution of this conjecture in the form of the $Tb$ theorems as stated above.
\end{remark}

I conclude the introduction by commenting briefly on the \(L^p\)-boundedness of Cauchy integrals, a fundamental question to measure the advances in the theory of singular integrals both in the scalar-valued and the vector-valued developments. The boundedness of the Cauchy integral on the circle, i.e., the Hilbert transform, is of course a classical theorem of M.~Riesz, and the extension of this result to the UMD-valued setting by Burkholder~\cite{Burkholder}, together with the converse statement by Bourgain~\cite{Bourgain:83}, may be considered the beginning of harmonic analysis in UMD spaces.

It was around the same time that the scalar-valued \(L^p\)-boundedness problem of the Cauchy integral on arbitrary Lipschitz graphs was answered positively by Coifman, McIntosh, and Meyer~\cite{CMM}. A few years later, this could be seen as a special case of the (homogeneous) \(Tb\) theorem due to David, Journ\'e, and Semmes~\cite{DJS}. The corresponding result in UMD spaces became available after Figiel proved his vector-valued \(T1\) theorem~\cite{Figiel:90}, since this bootstraps into \(Tb\) by the same trick as in \(Tb\) theorem~\ref{thm:TbL2} above.

Finally, a precise geometric characterization of the measures for which the associated Cauchy integral is bounded was identified in terms of a local curvature condition on \(\mu\) in an accumulation of efforts by several authors \cite{MMV,MelVer,NTV:Cauchy,Tolsa:Cauchy}. As before, this then became a corollary of the more general \(Tb\) theorems. The present results, once again, bring the vector-valued theory of the Cauchy integral to the same level.


\section{Strategy of the proof with historical remarks}

Among the large family of existing \(Tb\) theorems, there is no question about the parents of the present one: they are the non-homogeneous \(Tb\) theorem of Navarov, Treil and Volberg~\cite{NTV:Tb}, and my operator-valued \(Tb\) theorem~\cite{H:Tb}.
A reader familiar with the proof of either one of them will recognize much of the same general structure here, but in the details there are also substantial departures from the earlier approaches. This section gives an outline of the proof with commentary on the relation of its various parts to the existing arguments.

The proof starts from a ``twisted'' (or ``adapted'') martingale difference decomposition of the operator \(T\). Let \(\mathscr{D}=\bigcup_{k\in\Z}\mathscr{D}_k\) be a system of dyadic cubes in \(\R^N\): each subcollection \(\mathscr{D}_k\) is of the form
\begin{equation*}
  \mathscr{D}_k=\big\{x_k+2^k\big(n+[0,1)^N\big):n\in\Z^N\big\}
\end{equation*}
for some \(x_k\in\R^N\), and each \(Q\in\mathscr{D}_k\) is the exact union of \(2^N\) cubes \(Q'\in\mathscr{D}_{k-1}\). Let \(\Exp_k:=\Exp[\cdot|\sigma(\mathscr{D}_k)]\) be the associated conditional expectations which, because the \(\sigma\)-algebra \(\sigma(\mathscr{D}_k)\) is atomic, admit the explicit representation
\begin{equation*}
  \Exp_k f=\sum_{Q\in\mathscr{D}_k} \frac{1_Q}{\mu(Q)}\int_Q f\ud\mu.
\end{equation*}
If \(\mu(Q)=0\) for some cube, the term corresponding to \(Q\) in the above series may be simply taken to be zero. Note that, following~\cite{NTV:Tb}, the ``geometric'' indexing of the dyadic partitions \(\mathscr{D}_k\) is used, where larger \(k\) refers to larger cubes; this is different from the ``probabilistic'' indexing, where larger \(k\) refers to a finer \(\sigma\)-algebra and hence smaller generating cubes.

Given a para-accretive function \(b\), the \(b\)-twisted conditional expectations and their localized versions, for \(k\in\Z\) and \(Q\in\mathscr{D}_k\), are defined by
\begin{equation*}
  \Exp_k^b f:=b\frac{\Exp_k f}{\Exp_k b},\qquad
  \Exp_Q^b f:=1_Q\Exp_k^b f,
\end{equation*}
and the corresponding twisted martingale differences by
\begin{equation*}
  \D_k^b f:=\Exp_{k-1}^b f-\Exp_k^b f,\qquad
  \D_Q^b f:=1_Q\D_k^b f.
\end{equation*}
(In~\cite{H:Tb}, the adjoints of these operators are used instead, which does not make any essential difference.) By martingale convergence, there holds $\Exp_k^b f\to f$ pointwise a.e. and in $L^p(\mu)$ as $k\to-\infty$.

For any $m\in\Z$, it then follows that 
\begin{equation}\label{eq:twistedUC}
  f=\sum_{k\leq m}\D_k^b f+\Exp_m^b f
   =\sum_{\substack{Q\in\mathscr{D}\\ \ell(Q)\leq 2^m}}\D_Q^b f
      +\sum_{\substack{Q\in\mathscr{D}\\ \ell(Q)= 2^m}}\Exp_Q^b f
\end{equation}
with unconditional convergence in \(L^p(\mu;X)\) under the UMD assumption (see Section~\ref{sec:Haar} for details). So far everything is practically the same as in both \cite{H:Tb} and \cite{NTV:Tb}, with only minor technical differences. If $f$ is compactly supported, then for all $m$ large enough, depending only on the diameter of the support, the second sum on the right contains at most $2^N$ non-zero terms.

As the first departure from \cite{NTV:Tb}, but still quite closely following \cite{H:Tb}, the projections \(\D_Q^b\) and $\Exp_Q^b$ will be further represented in terms of rank-one operators as
\begin{equation}\label{eq:twistedHaar}
  \D_Q^b f=\sum_{u=1}^{2^N-1}b\varphi_{Q,u}^b\pair{\varphi_{Q,u}^b}{f},\qquad
  \Exp_Q^b f=b\varphi_{Q,0}^b\pair{\varphi_{Q,0}^b}{f},
\end{equation}
where quite precise information (established in Section~\ref{sec:Haar}) about the ``Haar'' functions \(\varphi_{Q,u}^b\) will be essential in deriving the required \(L^p\) bounds. There is a qualitative difference between the cancellative functions $\varphi_{Q,u}^b$ with $\int b\varphi_{Q,u}^b\ud\mu=0$, and the non-cancellative $\varphi_{Q,0}^b:=\big(\int_Q b\ud\mu\big)^{-1/2}1_Q$. Recall that the classical \(L^2\)-normalized Haar functions \(h_Q\) associated to a dyadic cube \(Q\) satisfy \(\Norm{h_Q}{1}=\abs{Q}^{1/2}\), \(\Norm{h_Q}{\infty}=\abs{Q}^{-1/2}\), and the equalities remain true up to constants even in the \(b\)-twisted case~\cite{H:Tb}. In the present situation, there is no upper control of the \(L^{\infty}(\mu)\) norm of \(\varphi_{Q,u}^b\) in terms of the measure \(\mu(Q)\), but this can be compensated by the smallness of the \(L^1(\mu)\) norm, so that the following important property still holds:
\begin{equation*}
  \Norm{\varphi_{Q,u}^b}{L^1(\mu)}\Norm{\varphi_{Q,u}^b}{L^{\infty}(\mu)}\lesssim 1.
\end{equation*}

To estimate the operator norm \(\Norm{T}{\bddlin(L^p(\mu;X))}\), a pairing \(\pair{g}{Tf}\) will be considered, where the compactly supported \(f\in L^p(\mu;X)\) and \(g\in L^{p'}(\mu;X^*)\) are expanded by means of \eqref{eq:twistedUC} and \eqref{eq:twistedHaar}, now taking one of the two para-accretive functions \(b_1\) and \(b_2\) from the assumptions of the \(Tb\) theorem in place \(b\):
\begin{equation}\label{eq:matrixT}
  \pair{g}{Tf}
  =\lim_{\epsilon\to 0}\sum_{\substack{Q\in\mathscr{D},R\in\mathscr{D}' \\ \epsilon\leq\ell(Q),\ell(R)\leq 2^m}}\sum_{u,v}
   \pair{g}{\varphi_{R,v}^{b_2}}\pair{\varphi_{R,v}^{b_2}b_2}{T(b_1\varphi_{Q,u}^{b_1})}
    \pair{\varphi_{Q,u}^{b_1}}{f},
\end{equation}
where the summation condition for $u$ is $u=1,\ldots,2^N-1$ for $\ell(Q)<2^m$ and $u=0,1,\ldots,2^N-1$ for $\ell(Q)=2^m$, and similarly for $v$ in terms of $R$. For a fixed $\epsilon>0$, the multiple summation consists of only finitely many non-zero terms, legitimating all the rearrangements that one may like to make in the course of the proof.

Following Nazarov, Treil and Volberg~\cite{NTV:Tb}, the functions \(f\) and \(g\) are expanded in terms of ``Haar'' functions related to two different dyadic systems \(\mathscr{D}\) and \(\mathscr{D}'\). An important aspect of the proof, already in~\cite{NTV:Tb} and even more decisively here, is the fact that parts of series~\eqref{eq:matrixT} cannot be directly controlled for a pair of preassigned dyadic systems, but only on average after taking the expectation over independent random choices of $\mathscr{D}$ and $\mathscr{D}'$---the underlying probability distribution is explained in Section~\ref{sec:random}. In this respect, it is useful to observe that the a priori continuity of $T$ ensures the uniform boundedness (involving the operator norm of $T$) of the truncated series in~\eqref{eq:matrixT}, so that the expectations may be moved in and out of the limit by dominated convergence.

This expansion of \(\pair{g}{Tf}\)---based on two independent multiresolution analyses of the domains of \(f\) and \(g\) (both of which are equal to \(\R^N\))---is essentially different from the one which Figiel introduced for the \(T1\) theorem in \cite{Figiel:90} and I adapted for \(Tb\) in \cite{H:Tb}. In Figiel's approach, a single multiresolution analysis of the product domain \(\R^N\times\R^N\) of \(f\otimes g\) was employed, which would mean that the summation over \(Q\in\mathscr{D}\) and \(R\in\mathscr{D}'\) comes with the restriction to cubes of the same size, \(\ell(Q)=\ell(R)\), while the summation range of \((u,v)\) is \(\{0,1,\ldots,2^N-1\}^2\setminus\{(0,0)\}\) on all the length-scales.

To simplify notation, the summations over the bounded ranges of $u$ and $v$ will mostly be suppressed, and I write
\begin{equation*}
  \varphi_Q:=\varphi^{b_1}_{Q,u},\qquad
  \psi_R:=\varphi^{b_2}_{R,v},\qquad
  T_{RQ}:=\pair{\psi_R b_2}{T(b_1\varphi_Q)}
\end{equation*}
for short.

As in \cite{NTV:Tb}, the analysis of the series in \eqref{eq:matrixT} will be divided into several cases depending on the relative size and position of the cubes \(Q,R\in\mathscr{D}\). By symmetry, it suffices to consider the half of the series with \(\ell(Q)\leq\ell(R)\). Modulo the extraction of appropriate \emph{paraproduct} operators (defined and treated in Section~\ref{sec:para}), the coefficients \(T_{RQ}\) exhibit good off-diagonal decay when the cubes \(Q\) and \(R\) move apart in the ``phase space'', where the coordinates are the spatial position and the size of a cube. Thanks to this decay, it is possible (in Section~\ref{sec:separated}) to separately treat countably many subseries of \eqref{eq:matrixT}, a typical one consisting of cubes such that \(\dist(Q,R)\sim 2^j\ell(R)\) and \(\ell(Q)=2^{-n}\ell(R)\), and the decay will provide estimates which allows to make the final summation over \(j,n=0,1,2,\ldots\) with absolute convergence. A further separate treatment is made for cubes of which one contains the other (deeply) in its interior (Section~\ref{sec:contained}), and yet another for cubes of essentially the same size and very close or even touching each other (Section~\ref{sec:close}).

In each case, the subseries in question (consisting of \(R\in\mathscr{D}'\) and \(Q\) from some subcollection \(\mathscr{D}(R)\subset\mathscr{D}\), depending on \(R\)) is first estimated by the following basic randomization trick:

\begin{lemma}
The following inequality holds:
\begin{equation}\label{eq:firstBy}\begin{split}
  &\Babs{\sum_{R\in\mathscr{D}'}\pair{g}{\psi_R}\sum_{Q\in\mathscr{D}(R)}T_{RQ}\pair{\varphi_Q}{f}} \\
  &\lesssim\Norm{g}{L^{p'}(\mu;X^*)}\BNorm{\sum_{k\in\Z}\radem_k
  \sum_{R\in\mathscr{D}_k'}\psi_R(x)\sum_{Q\in\mathscr{D}(R)}T_{RQ}\pair{\varphi_Q}{f}}{L^p(\prob\otimes\mu;X)},
\end{split}\end{equation}
\end{lemma}

\begin{proof}
The construction of the Haar functions gives the identity $\int b_2\psi_R^2\ud\mu=1$. Using this in the first step, one estimates
\begin{equation*}\begin{split}
  LHS\eqref{eq:firstBy}
  &=\Big|\iint_{\Omega\times\R^N}\sum_{S\in\mathscr{D}'}\radem_S \pair{g}{\psi_S}(b_2\psi_S)(x) \\
  &\qquad
          \times\sum_{R\in\mathscr{D}'}\radem_R\psi_R(x)\sum_{Q\in\mathscr{D}(R)}T_{RQ}\pair{\varphi_Q}{f}
    \ud\prob(\radem)\ud\mu(x)\Big| \\
  &\leq\BNorm{\sum_{S\in\mathscr{D}'}\radem_S\pair{g}{\psi_S}b_2\psi_S}{L^{p'}(\prob\otimes\mu;X^*)} \\
  &\qquad
       \times\BNorm{\sum_{R\in\mathscr{D}'}\radem_R\psi_R(x)\sum_{Q\in\mathscr{D}(R)}T_{RQ}\pair{\varphi_Q}{f}}{
     L^p(\prob\otimes\mu;X)} 
  \lesssim RHS\eqref{eq:firstBy},
\end{split}\end{equation*}
where the final estimate for the first factor is an application of the unconditionality of \eqref{eq:twistedUC} in \(L^{p'}(\mu;X^*)\); in the second factor, the basic observation is made that, for a fixed \(x\in\R^N\), the summation over \(R\in\mathscr{D}'=\bigcup_{k\in\Z}\mathscr{D}_k'\) only contains one non-zero \(\psi_R(x)\) for each \(\mathscr{D}_k'\), and hence it does not matter if the random signs are indexed by cubes or the size of the cubes. (This observation will be repeatedly applied without further notice.)
\end{proof}

The collections \(\mathscr{D}(R)\) are always of such a form that \(R\in\mathscr{D}_k'\) implies \(\mathscr{D}(R)\subset\mathscr{D}_{k-n}\) for some \(n\in\N\), independent of \(R\). Also, when \(Q\in\mathscr{D}(R)\), the cube \(R\) will be contained in a dyadic ancestor \(Q^{(n+a)}\) of \(Q\), where $a=a(j)$ grows linearly in $j$ (recall that we are considering cubes with $\dist(Q,R)\sim 2^j\ell(R)$), but for a technical reason with a slope slightly bigger than $1$.
The quantity in the \(L^p(\mu;X)\) norm to be estimated is hence of the form
\begin{equation*}
  \sum_{k\in\Z}\radem_k\sum_{S\in\mathscr{D}_{k+a}}
    1_S(x)\int_{S} K_S(x,y)\D_{k-n}^{b_1}f(y)\ud\mu(y)
  =:\sum_{k\in\Z}\radem_k T^{(k)}\D_{k-n}^{b_1}f(x),
\end{equation*}
and it remains to prove that
\begin{equation}\label{eq:genericToProve}
  \Exp\BNorm{\sum_{k\in\Z}\radem_k T^{(k)}\D_{k-n}^{b_1}f}{L^p(\mu;X)}
  \lesssim 2^{-(n+j)\sigma}\Exp\BNorm{\sum_{k\in\Z}\radem_k \D_{k-n}^{b_1}f}{L^p(\mu;X)},
\end{equation}
since this is bounded by \(2^{-(n+j)\sigma}\Norm{f}{L^p(\mu;X)}\) due to the unconditionality, and the exponential factor allows the summation over \(n,j\in\N\) to complete the estimate of the full series \eqref{eq:matrixT}.

The integral kernels \(K_S\) will typically satisfy bounds of the type \(\Norm{K_S}{\infty}\lesssim 2^{-(n+j)\sigma}\ell(S)^{-d}\), with similar but somewhat more complicated form when \(j\in\{0,1\}\), i.e., when the cubes \(Q\) are close to or contained inside \(R\). Getting these estimates requires the fine properties of the ``Haar'' functions \(\varphi_Q\) and \(\psi_R\). Recalling that \(\mu(S)\lesssim \ell(S)^d\), it is seen that \(2^{(n+j)\sigma}T^{(k)}F(x)\) is a weighted average of \(F\) in a neighbourhood of \(x\).

In the classical Calder\'on--Zygmund theory, such averaging operators were usually controlled by the Hardy--Littlewood maximal operator \(M\), and the estimate \eqref{eq:genericToProve} could be deduced from the Fefferman--Stein square-function estimate for \(M\). The lack of a comparable vector-valued theory of a maximal function has necessitated the invention of alternative tools to circumvent the maximal function arguments in the estimation of integral operators.

A powerful substitute was provided by Bourgain's square function estimate \cite{Bourgain:86} for the translations \(\tau_y:h\mapsto h(\cdot+y)\), which can be viewed as the basic building blocks of integral operators via the formula
\begin{equation}\label{eq:intopTrans}\begin{split}
  &\int_{S}K_S(x,y)f(y)\ud\mu(y) \\
  &=\int_{B(0,C)} K_S(x,x+\ell(S)u)(\tau_{\ell(S)u}f)(x)\ud\mu(x+\ell(S)u).
\end{split}\end{equation}
(Note that this simlifies for the Lebesgue measure \(\ud\mu(y)=\ud y\), since then \(\ud\mu(x+\ell(S)u)=\ell(S)^N\ud u\), so that the integrations on \(B(0,C)\) can be carried out with respect to a fixed reference measure.)
Bourgain showed that
\begin{equation}\label{eq:Bourgain}
  \BNorm{\sum_{j\in\Z}\radem_j\tau_{2^j y}f_j}{L^p(\Omega\times\R^N;X)}
  \lesssim\log(2+\abs{y})\BNorm{\sum_{j\in\Z}\radem_j f_j}{L^p(\Omega\times\R^N;X)}
\end{equation}
(where \(\R^N\) is equipped with the Lebesgue measure) assuming that the Fourier transforms of the \(f_j\) are restricted by the condition \(\supp\hat{f}_j\subseteq B(0,2^{-j})\)---a condition which is naturally satisfied when these functions arise from a Littlewood--Paley-type decomposition. Figiel~\cite{Figiel:88} gave a variant of this result where it is required instead that \(f_j=\Exp_j f_j\) and \(y\in\Z^N\), which would be closer to the present martingale setting. (The original formulation in \cite{Figiel:88} in terms of the Haar functions is slightly different but the equivalence is immediate.)

All the known Banach space -valued \(T1\) and \(Tb\) theorems so far have been based on one of these two remarkable results: Figiel's \(T1\)~\cite{Figiel:90} and my \(Tb\)~\cite{H:Tb} on the martingale version, and the \(T1\) theorem of mine and Weis \cite{HW:T1} on the Fourier-analytic one. However, a moment's thought reveals that there is no hope of extending the translation techniques to the non-homogeneous situation. Since only an upper control of the measure of balls is assumed, a small translation of just a single function (not to mention a sequence of functions as above) may result in its support being moved from a set of negligible measure to one with a large \(\mu\)-mass, with uncontrollable effect on the \(L^p\) norm.

To overcome this problem, I use a different trick based on a two-sided inequality for so-called tangent martingale difference sequences due to McConnell~\cite{McConnell}. This is a stochastic decoupling estimate, explained in detail in Section~\ref{sec:trick}, which McConnell originally employed for the construction of It\^o-type integrals of UMD-valued random processes. Thus the trick itself is not new, but it seems not to have been exploited in the context of Calder\'on--Zygmund theory before. Although it still avoids maximal functions, this method is somewhat closer in spirit to the classical maximal function techniques than the translation inequalities~\eqref{eq:Bourgain}, which have been the most refined tools in vector-valued harmonic analysis for the past twenty years. I expect this trick to find further applications besides the results of the present paper. Indeed, after originally writing this prophesy, I already discovered one such application (new even for scalar-valued functions) in the context of pseudo-localization of singular integral operators \cite{H:pseudoloc}, and there should be more.

This concludes the historical--strategic overview, and I now turn to the details.


\section{A Carleson embedding theorem}

This section provides a ``Carleson-type'' embedding theorems, which will play a r\^ole both in establishing the unconditionality of the twisted martingale difference decomposition~\eqref{eq:twistedUC} in the next section, and later on in handling the paraproduct parts of the operator~\(T\). The result will be formulated in an abstract filtered space setting, since the special case of actual interest involving \(\R^N\) with its systems of dyadic cubes would not provide any simplification and could at most distract the attention from the measure-theoretic core of the arguments.

Let \((E,\mathscr{M},\mu)\) be a \(\sigma\)-finite measure space. Let \(\vec{\mathscr{F}}=(\mathscr{F}_j)_{j\in\Z}\) be a decreasing sequence of sub-\(\sigma\)-algebras of \(\mathscr{M}\), i.e., \(\mathscr{F}_{j-1}\supseteq\mathscr{F}_j\), such that each \((E,\mathscr{F}_j,\mu)\) is also \(\sigma\)-finite. The short hand notation \(\Exp_j:=\Exp[\cdot|\mathscr{F}_j]\) will be used for the corresponding conditional expectations.
Let \(\mathscr{F}_j^+\) consist of the sets \(A\in\mathscr{F}_j\) of finite positive measure.

Given a sequence of functions \(\theta_j:E\to X_1\), such that \(1_A\theta_j\in L^1(E;X_1)\) for all \(A\subseteq E\) of finite measure, the following Carleson norms were introduced by McIntosh, Portal, and the author \cite{HMP} in a special case:
\begin{equation*}\begin{split}
  \Norm{\{\theta_j\}_{j\in\Z}}{\Car^p(\vec{\mathscr{F}};X_1)}
  &:=\sup_{k\in\Z}\BNorm{\Big(\Exp_k
    \BNorm{\sum_{j\leq k}\radem_j \theta_j}{L^p(\Omega;X_1)}^p\Big)^{1/p}}{L^{\infty}(E)} \\
  &=\sup_{k\in\Z}\sup_{A\in\mathscr{F}_k^+}\mu(A)^{-1/p}\BNorm{1_A\sum_{j\leq k}\radem_j \theta_j}{L^p(\Omega\times E;X_1)}.
\end{split}\end{equation*}

\begin{proposition}\label{prop:JN}
If \(\theta_j=\Exp_j \theta_j\) for all \(j\in\Z\), then the Carleson norms
\begin{equation*}
  \Norm{\{\theta_j\}_{j\in\Z}}{\Car^p(\vec{\mathscr{F}};X_1)}
\end{equation*}
are equivalent for all \(p\in[1,\infty)\).
\end{proposition}

\begin{proof}
This could be proven in a similar way as the well-known equivalence of the different (martingale) \(\BMO^p\) norms. Instead, I will show how to reduce the claim to the mentioned result. By approximation, it suffices to treat finitely non-zero sequences \(\theta_j\) in order to avoid problems of convergence in the following expressions.

Consider the filtration $\vec{\mathscr{G}}=(\mathscr{G}_k)_{k\in\Z}$  on \(\Omega\times E\), defined by $\mathscr{G}_k:=\sigma(\mathscr{E}_k,\mathscr{F}_k)$, where \(\mathscr{E}_k:=\sigma(\radem_j;j\geq k)\). Let $\tilde{\Exp}_k:=\Exp[\cdot|\mathscr{G}_k]$. The space $\BMO^p(\mathscr{G}_k;X_1)$ consists of all $\theta:E\times\Omega\to X_1$, integrable over sets of finite measure, such that
\begin{equation*}
\begin{split}
  \Norm{\theta}{\BMO^p}
  &:=\sup_{k\in\Z}\bNorm{(\tilde{\Exp}_{k}\norm{\theta-\tilde{\Exp}_{k+1}\theta}{X_1}^p)^{1/p}}{\infty} \\
  &=\sup_{k\in\Z}\sup_{\tilde{A}\in\mathscr{G}_k^+}
      \tilde\mu(\tilde{A})^{-1/p}\Norm{1_{\tilde{A}}[\theta-\tilde{\Exp}_{k+1}\theta]}{p},\qquad
      \tilde\mu:=\mu\times\prob,
\end{split}
\end{equation*}
is finite.

Now consider the particular function \(\theta:=\sum_{j\in\Z}\radem_j \theta_j\).  Then
\begin{equation*}
  \tilde{\Exp}_k\theta=\sum_{j\geq k}\radem_j \theta_j,\qquad
  \theta-\tilde{\Exp}_k\theta=\sum_{j< k}\radem_j \theta_j,
\end{equation*}
and, by the tower rule for conditional expectations with respect to $\mathscr{G}_k\subseteq\sigma(\mathscr{E}_k,\mathscr{M})$,
\begin{equation*}
  \tilde{\Exp}_k\norm{\theta-\tilde{\Exp}_{k+1}\theta}{X_1}^p
  =\tilde{\Exp}_k\Exp\Big[\Bnorm{
     \sum_{j\leq k}\radem_j \theta_j}{X_1}^p\Big|\sigma(\mathscr{E}_k,\mathscr{M})\Big].
\end{equation*}
The conditional expectation inside is computed by keeping the variables \(\radem_k\) and \(x\in E\) fixed and taking the average over all \(\radem_j\) for \(j<k\). Writing \(\radem_j'\) for another set of independent random signs and \(\Exp'\) for the corresponding expectation, this quantity can be written as
\begin{equation*}
  \Exp'\Bnorm{\sum_{j<k}\radem_j' \theta_j+\radem_k \theta_k}{X_1}^p
  =  \Exp'\Bnorm{\sum_{j\leq k}\radem_j' \theta_j}{X_1}^p,
\end{equation*}
where the equality follows from the observation that the first expectation is actually independent of the sign \(\radem_k\).
Hence
\begin{equation*}
  \tilde{\Exp}_k\norm{\theta-\tilde{\Exp}_{k+1}\theta}{X_1}^p
  =\Exp_k\Exp'\Bnorm{\sum_{j\leq k}\radem_j' \theta_j}{X_1}^p
  =\Exp_k\BNorm{\sum_{j\leq k}\radem_j \theta_j}{L^p(\Omega;X_1)}^p.
\end{equation*}
Applying \(\sup_{k\in\Z}\Norm{(\cdot)^{1/p}}{L^{\infty}(E)}\) to the left side above, one gets the martingale \(\BMO^p\) norm of \(\theta\), while the same functional of the right side yields \(\Norm{\{\theta_j\}_{j\in\Z}}{\Car^p(\vec{\mathscr{F}};X_1)}\). The equivalence of the \(\Car^p\) norms thus follows from the equivalence of the (martingale) \(\BMO^p\) norms, which is the well-known John--Niren\-berg inequality.
\end{proof}

\begin{remark}
With a one-point measure space \(E=\{e\}\) and \(\theta_j=\xi_j\in X_1\), it follows that \(\Norm{\{\theta_j\}_{j\in\Z}}{\Car^p}=\Norm{\sum\radem_j\xi_j}{L^p(\Omega;X_1)}\). Hence the previous proof shows that Kahane's inequality (the equivalence of the different \(L^p\) norms of such random sums; \cite{DJT}, Theorem~11.1) is a consequence of the martingale John--Nirenberg inequality. This is probably known to experts, but I did not encounter this observation before.
\end{remark}

Suppose that there are three Banach spaces \(X_1,X_2,X_3\) with \(X_2\subseteq\bddlin(X_1,X_3)\); the point is here that \(X_2\) may be required to have some properties which the full operator space \(\bddlin(X_1,X_3)\) would almost never satisfy.
Given a sequence \(\{\theta_j\}_{j\in\Z}\in\Car^1(\vec{\mathscr{F}};X_2)\), the ``paraproduct type'' operator 
\begin{equation}\label{eq:Pf}
  Pf:=\sum_{j\in\Z}\radem_j \theta_j\Exp_j f,
\end{equation}
acting on \(f\in L^p(E;X_1)\), is of interest.

There are two closely related results which guarantee the boundedness of \(P\) from \(L^p(E;X_1)\) to \(L^p(\Omega\times E;X_3)\):

\begin{theorem}\label{thm:Pf}
Let \(X_3\) be a UMD space, and \(1<p<\infty\). Let \(\{\theta_j\}_{j\in\Z}\) be a sequence such that \(\theta_j=\Exp_j \theta_j\) for all \(j\in\Z\). Then \(P\) defined in \eqref{eq:Pf} satisfies
\begin{equation*}
  \Norm{P f}{L^p(\Omega\times E;X_3)}
  \lesssim\Norm{\{\theta_j\}_{j\in\Z}}{\Car^1(\vec{\mathscr{F}};X_2)}\Norm{f}{L^p(E;X_1)}.
\end{equation*}
\end{theorem}

\begin{theorem}\label{thm:PfRMF}
Let \(X_1\) be an RMF space, \(1<p<\infty\), and \(\eta>0\). Let the unit-ball \(\bar{B}_{X_2}\) of \(X_2\) be a Rademacher-bounded subset of \(\bddlin(X_1,X_3)\).  Then \(P\) defined in \eqref{eq:Pf} satisfies
\begin{equation*}
  \Norm{P f}{L^p(\Omega\times E;X_3)}
  \lesssim\Norm{\{\norm{\theta_j(\cdot)}{X_2}\}_{j\in\Z}}{\Car^{p+\eta}(\vec{\mathscr{F}})}\Norm{f}{L^p(E;X_1)}.
\end{equation*}
\end{theorem}

Although it played an important r\^ole in earlier versions of this paper,  Theorem~\ref{thm:PfRMF} is eventually not needed here, and it is only recorded above for reasons of comparison. As stated, it is a slight generalization of Theorem 8.2 from my paper with McIntosh and Portal \cite{HMP} and may be proven by an adaptation of the same argument. In an earlier version of this paper, I had tried to push the analogy of Theorems \ref{thm:Pf} and \ref{thm:PfRMF} a bit too far by attempting to deduce even the former one by a variation of the same technique. This argument turned out to be flawed, and I am now unaware of any method of proof which would give both theorems as applications of a common general principle.

The proof of Theorem~\ref{thm:Pf}, which I will give, follows a similar approach as the corresponding results behind the earlier vector-valued $Tb$ theorems, which goes back to Bourgain (see Figiel and Wojtaszczyk \cite{FigWoj}, who attribute a key step of their argument to him). It relies on interpolation between appropriate $H^1$--$L^1$ and $L^{\infty}$--$\BMO$ estimates, where the martingale versions of these spaces will be relevant. Recall that the martingale Hardy space $H^1(\vec{\mathscr{F}};X_1)$ consists of the functions $f\in L^1(E;X)$ with $Mf:=\sup_{k\in\Z}\norm{\Exp_k f(\cdot)}{X_1}\in L^1(E)$. For the present purposes, however, it will be convenient to use the characterization of this space, due to Herz \cite{Herz}, as
\begin{equation*}
  H^1(\vec{\mathscr{F}};X_1)={}^{r}L^{1}(\vec{\mathscr{F}};X_1)+AC(\vec{\mathscr{F}};X_1),\qquad
  1<r\leq\infty.
\end{equation*}
Here the space ${}^{r}L^{1}(\vec{\mathscr{F}};X_1)$ of \emph{$L^r$-regulated $L^1$-functions} consists of all $f\in L^1(E;X_1)$ with a representation $f=\sum_{j=1}^{\infty}\lambda_j a_j$, where $(\lambda_j)_{j=1}^{\infty}\in\ell^1$ and each $a_j\in L^r(E;X_1)$ is an \emph{atom} of $L^r$-type. This means that there exists $k=k(j)\in\Z$ and $A\in\mathscr{F}_k$ such that $a_j=1_A a_j$, $\Exp_k a=0$, and $\mu(A)^{1/r'}\Norm{a}{r}\leq 1$. The norm in this space is the infimum of the $\ell^1$ norms of the coefficient sequences over all such representations. The space $AC(\vec{\mathscr{F}};X_1)$ of \emph{absolutely convergent} $L^1$-martingales consists of $h=\sum_{k\in\Z}\D_k h\in L^1(E;X_1)$ with $\Norm{h}{AC}:=\sum_{k\in\Z}\Norm{\D_k h}{1}<\infty$. This latter component of $H^1(\vec{\mathscr{F}};X_1)$ is required, in addition to the atomic part familiar from the classical theory, because of the non-doubling nature of the underlying measure.

For convenience, it will be assumed that the sequence $\{\theta_j\}_{j\in\Z}$ is finitely nonzero, but the bounds will be proven in terms of its Carleson norm only. Then it is straightforward to pass to the general case in the final $L^p$ estimates of interest.
The proof begins with:

\begin{lemma}
Let $X_3$ be a UMD space and $r\in(1,\infty]$. Then
\begin{equation*}
  \Norm{Pf}{L^1(E\times\Omega;X_3)}
  \lesssim\Norm{\{\theta_j\}_{j\in\Z}}{\Car^1(\vec{\mathscr{F}};X_2)}
    \Norm{f}{{}^r L^1(\vec{\mathscr{F}};X_1)}.
\end{equation*}
\end{lemma}

\begin{proof}
Thanks to the a priori boundedness under the assumption that $\{\theta_j\}_{j\in\Z}$ is finitely non-zero, it suffices to prove the uniform bound on all atoms $a$. So let $a=1_A a$ with $A\in\mathscr{F}_k$, $\Exp_k a=0$ and $\mu(A)^{1/r'}\Norm{a}{r}\leq 1$. Let also the Carleson norm be normalized to be $1$. Choose auxiliary exponents $p\in(1,r)$ and $q\in[p,\infty)$ such that $1/p=1/q+1/r$. Now $\Exp_j a=\Exp_j\Exp_k a=0$ for $j\geq k$, while $\theta_j\Exp_j a=1_A\Exp_j(\theta_j a)$ for $j<k$, and hence 
\begin{equation*}
  Pa=\sum_{j<k}\radem_j 1_A\Exp_j(\theta_j a).
\end{equation*}
This leads to the estimate
\begin{align*}
  \Norm{Pa}{1}
  &\leq\mu(A)^{1/p'}\BNorm{\sum_{j<k}\radem_j\Exp_j(\theta_j a)}{p}
    \lesssim\mu(A)^{1/p'}\BNorm{\sum_{j<k}\radem_j\theta_j a}{p} \\
   &\leq\mu(A)^{1/p'}\BNorm{\sum_{j<k}\radem_j\theta_j 1_A}{q}\Norm{a}{r}
     \leq\mu(A)^{1/p'}\mu(A)^{1/q}\mu(A)^{-1/r'}=1. 
\end{align*}
where the second step was Bourgain's vector-valued Stein inequality~\cite{Bourgain:86}, and all the other bounds are elementary.
\end{proof}

The $H^1(\vec{\mathscr{F}};X_1)\to L^1(E\times\Omega;X_3)$ boundedness of $P$ is completed by the characterization of Herz together with:

\begin{lemma}
For arbitrary Banach spaces,
\begin{equation*}
    \Norm{Pf}{L^1(E\times\Omega;X_3)}
  \lesssim\Norm{\{\theta_j\}_{j\in\Z}}{\Car^1(\vec{\mathscr{F}};X_2)}
    \Norm{f}{AC(\vec{\mathscr{F}};X_1)}.
\end{equation*}
\end{lemma}

\begin{proof}
Notice that $\Exp_j\D_k=\D_k$ for $j<k$ and zero otherwise. Hence
\begin{equation*}
  P\D_k f=\sum_{j<k}\radem_j\theta_j\D_k f,
\end{equation*}
and thus (taking the Carleson norm equal to one again)
\begin{align*}
  \Norm{P\D_k f}{1}
  &=\int_{E\times\Omega}\tilde{\Exp}_{k-1}\Bnorm{\sum_{j<k}\radem_j\theta_j\D_k f}{X_3}\ud\tilde\mu \\
  &\leq\int_{E\times\Omega}\Big(\tilde{\Exp}_{k-1}\Bnorm{\sum_{j<k}\radem_j\theta_j}{X_2}\Big)
      \norm{\D_k f}{X_1}\ud\tilde\mu \\
   &\leq\int_{E\times\Omega}\norm{\D_k f}{X_1}\ud\tilde\mu=\Norm{\D_k f}{1},
\end{align*}
where the second step follows from the fact that $\D_k f$ is already $\mathscr{G}_{k-1}$-measurable, while in the third one it was observed that the quantity in parentheses is uniformly bounded by the $\Norm{\{\theta_j\}_{j\in\Z}}{\Car^1}$. Summing over $k\in\Z$ completes the proof.
\end{proof}

In the upper end, the space $\BMO(\vec{\mathscr{G}};X_3)$ on the product measure space $E\times\Omega$, as defined in the proof of Proposition~\ref{prop:JN}, is needed.

\begin{lemma}
Let $X_3$ be a UMD space. Then
\begin{equation*}
  \Norm{Pf}{\BMO(\vec{\mathscr{G}};X_3)}
  \lesssim\Norm{\{\theta_j\}_{j\in\Z}}{\Car^1(\vec{\mathscr{F}};X_2)}
    \Norm{f}{L^{\infty}(E;X_1)}.
\end{equation*}
\end{lemma}

\begin{proof}
As in the proof of Proposition~\ref{prop:JN}, there holds
\begin{equation*}
  (I-\tilde{\Exp}_k)Pf=\sum_{j<k}\radem_j\Exp_j(\theta_j f).
\end{equation*}
Let $\tilde{A}\in\mathscr{G}_{k-1}$, and notice that the signs $\radem_j$, for $j<k-1$, are independent of $\mathscr{G}_{k-1}$ as well as, obviously, of the functions $\Exp_j(\theta_j f)$. Hence they may be replaced by independent copies $\radem_j'$ on another probability space $\Omega'$ as far as the computation of norms is concerned. This leads to
\begin{align*}
  &\Norm{1_{\tilde{A}}(I-\tilde{\Exp}_k)Pf}{L^p(E\times\Omega;X_3)} \\
  &=\BNorm{1_{\tilde{A}}\Big(\radem_{k-1}\Exp_{k-1}(\theta_{k-1}f)
       +\sum_{j<k-1}\radem_j'\Exp_j(\theta_j f)\Big)}{L^p(E\times\Omega;L^p(\Omega';X_3))}.
\end{align*}
The observation that the inner $L^p(\Omega';X_3)$ norm is actually independent of the value of $\radem_{k-1}\in\{-1,+1\}$ allows even its replacement by an independent copy, resulting in
\begin{equation*}
  \Norm{1_{\tilde{A}}(I-\tilde{\Exp}_k)Pf}{L^p(E\times\Omega;X_3)}
  =\BNorm{1_{\tilde{A}}
       \sum_{j<k}\radem_j'\Exp_j(\theta_j f)}{L^p(E\times\Omega;L^p(\Omega';X_3))}.
\end{equation*}

By Fubini's theorem, $\tilde{A}(\omega):=\{x\in E:(x,\omega)\in\tilde{A}\}\in\mathscr{F}_{k-1}$ for $\prob$-a.e.~$\omega\in\Omega$, so that the multiplication operator with $1_{\tilde{A}(\omega)}$ commutes with $\Exp_j$ for $j<k$. Thus
\begin{align*}
  &\BNorm{1_{\tilde{A}}
       \sum_{j<k}\radem_j'\Exp_j(\theta_j f)}{L^p(E\times\Omega;L^p(\Omega';X_3))}^p \\
   &=\int_{\Omega}\BNorm{\sum_{j<k}\radem_j'\Exp_j(1_{\tilde{A}(\omega)}\theta_j f)}{
           L^p(E\times\Omega;X_3)}^p\ud\prob(\omega) \\
  &\lesssim\int_{\Omega}\BNorm{\sum_{j<k}\radem_j' 1_{\tilde{A}(\omega)}\theta_j f}{
           L^p(E\times\Omega;X_3)}^p\ud\prob(\omega) \\
   &\leq\int_{\Omega}\BNorm{1_{\tilde{A}(\omega)}\sum_{j<k}\radem_j' \theta_j}{
          L^p(E\times\Omega;X_2)}^p\Norm{f}{\infty}\ud\prob(\omega) \\
   &\leq\int_{\Omega}\mu(\tilde{A}(\omega))\ud\prob(\omega)\Norm{f}{\infty}
     =\tilde\mu(\tilde{A})\Norm{f}{\infty},
\end{align*}
where the first estimate was Bourgain's vector-valued Stein inequality, the third one the assumed Carleson condition, while the second is obvious. This completes the proof.
\end{proof}

Taken together, the last three lemmas yield:

\begin{proof}[Proof of Theorem~\ref{thm:Pf}]
It has been shown that
\begin{equation*}
  P:H^1(\vec{\mathscr{F}};X_1)\to L^1(E\times\Omega;X_3),\quad
  P:L^{\infty}(E;X_1)\to \BMO(\vec{\mathscr{G}};X_3);
\end{equation*}
thus the composition of $P$ with the sharp maximal operator $M^{\#}$ (more precisely, its martingale version with respect to $\vec{\mathscr{G}}$) maps
\begin{equation*}
  M^{\#}P:H^1(\vec{\mathscr{F}};X_1)\to L^{1,\infty}(E\times\Omega),\quad
  M^{\#}P:L^{\infty}(E;X_1)\to L^{\infty}(E\times\Omega),
\end{equation*}
and hence the boundedness of
\begin{equation*}
  M^{\#}P:L^p(E;X_1)\to L^p(E\times\Omega),\quad
  P:L^p(E;X_1)\to L^p(E\times\Omega;X_3)
\end{equation*}
follow from standard interpolation results.
\end{proof}

As a matter of fact, the preceding proof of Theorem~\ref{thm:Pf} was essentially written down in my original manuscript of \cite{H:Tb} already, but as this generality was not necessary for the version of the $Tb$ theorem then under consideration, the referee insisted in leaving it out in favour of a simpler argument valid for doubling measures only, and hence it did not appear in the published version of \cite{H:Tb}.

\section{Martingale difference decomposition}\label{sec:Haar}

In this section I prove the unconditional convergence of the twisted martingale difference decomposition stated in \eqref{eq:twistedUC} and~\eqref{eq:twistedHaar} and establish the basic properties of the ``Haar'' functions \(\phi_Q,\psi_R\) appearing in this decomposition. By standard considerations involving duality and the density in \(L^p(\mu)\) of linear combinations of indicators of dyadic cubes, it suffices for~\eqref{eq:twistedUC} to show the following randomized unconditionality estimate. In the doubling case, it was proven in~\cite{H:Tb}.

\begin{proposition}\label{prop:UC}
\begin{equation*}
  \BNorm{\sum_{k\in\Z}\radem_k\D_k^b f}{L^p(\prob\otimes\mu;X)}
  \lesssim\Norm{f}{L^p(\mu;X)}.
\end{equation*}
\end{proposition}

\begin{proof}
Write out
\begin{equation}\label{eq:writeOut}
  \frac{1}{b}\D_k^b f
  =\frac{\Exp_{k-1}f}{\Exp_{k-1}b}-\frac{\Exp_k f}{\Exp_k b}
  =-\frac{\D_k b}{\Exp_k b\cdot\Exp_{k-1}b}\Exp_{k-1}f
   +\frac{\D_k f}{\Exp_k b},
\end{equation}
and observe that the factors \(b\) and \(\Exp_k b\) may be discarded by their boundedness from above and below. The second term on the right above is then simply a martingale difference of \(f\), so their random sum is estimated as a direct application of UMD.

The random sum of the first terms gives the paraproduct \(Pf\) from~\eqref{eq:Pf} with \(\theta_k:=\D_{k+1}b=\Exp_k \theta_k\), and hence is dominated by \(\Norm{f}{L^p(\mu;X)}\) times the supremum over \(k\in\Z\) and \(A\in\mathscr{F}_k^+\) of
\begin{equation*}
  \mu(A)^{-1/q}\Big\{\BNorm{\sum_{j<k}\radem_j\D_{j+1}(1_A b)}{L^q(\prob\otimes\mu)}
    +\Norm{1_A\D_{k+1}b}{L^q(\mu)}\Big\},\qquad q\in(p,\infty).
\end{equation*}
The first term in braces is bounded by \(C\Norm{1_A b}{L^q(\mu)}\leq C\mu(A)^{1/q}\Norm{b}{\infty}\) using the UMD property of \(\C\), while the second one is dominated by \(\mu(A)^{1/q}\big(\Norm{\Exp_k b}{\infty}+\Norm{\Exp_{k+1}b}{\infty}\big)\leq 2\mu(A)^{1/q}\Norm{b}{\infty}\), since the conditional expectations are contractions in \(L^{\infty}(E)\). Collecting everything together, the proof is complete.
\end{proof}

I then pass to the finer decomposition of the martingale differences \(\D_k^b f\) in terms of rank-one operators.
Generalizing the notation \(\Exp^b_Q\) for \(Q\in\mathscr{D}\), denote
\begin{equation*}
  \Exp^b_A f:=1_A\frac{\int_A f\ud\mu}{\int_A b\ud\mu}\cdot b,
\end{equation*}
when \(A\) is any measurable set with \(\int_A b\ud\mu\neq 0\); the special case \(b\equiv 1\) will be abbreviated as \(\Exp_A:=\Exp_A^1\). If \(\mathscr{A}\) is a disjoint collection of such sets, write
\begin{equation*}
  \Exp^b_{\mathscr{A}}f:=\sum_{A\in\mathscr{A}}\Exp^b_A f.
\end{equation*}
With this notation one can express
\begin{equation*}
  \D^b_Q f
  =\Exp^b_{\{Q'\in\mathscr{D};Q'\subset Q,\ell(Q')=\ell(Q)/2\}}f-\Exp^b_Q f.
\end{equation*}

\begin{lemma}
For each \(Q\in\mathscr{D}\), its \(2^N\) subcubes \(Q_u\in\mathscr{D}\) with \(\ell(Q_u)=\ell(Q)/2\) may be indexed in such a way that
\begin{equation}\label{eq:subaccretive}
  \Babs{\int_{\bigcup_{u=k}^{2^N}Q_u} b\ud\mu}\geq
  [1-(k-1)2^{-N}]\delta\mu(Q)
\end{equation}
for all \(k=1,\ldots,2^N\).
\end{lemma}

\begin{proof}
The case \(k=1\) is fine for any ordering of the subcubes. Let us assume that we have an indexing of the cubes \(Q_1,\ldots,Q_{j-1}\) so that \eqref{eq:subaccretive} holds for all \(k=1,\ldots,j<2^N\). In particular
\begin{equation*}\begin{split}
  [1-(j-1)2^{-N}]\delta\mu(Q) &\leq\Babs{\sum_{u=j}^{2^N}\int_{Q_u}b\ud\mu}
  =\frac{1}{2^N-j}\Babs{\sum_{u=j}^{2^N}\sum_{\ell=j,\ell\neq u}^{2^N}\int_{Q_{\ell}}b\ud\mu} \\
  &\leq\frac{2^N-(j-1)}{2^N-j}\max_{u\in\{j,\ldots,2^N\}}\Babs{\int_{\bigcup_{\ell=j}^{2^N}Q_{\ell}\setminus Q_u}b\ud\mu}.
\end{split}\end{equation*}
It follows that for at least one \(u\in\{j,\ldots,2^N\}\), we have
\begin{equation*}
  \Babs{\int_{\bigcup_{\ell=j}^{2^N}Q_{\ell}\setminus Q_u}b\ud\mu}\geq[1-j2^{-N}]\delta\mu(Q).
\end{equation*}
By reordering the remaining cubes, we may assume that \(u=j\), and then we have fixed an indexing of the cubes \(Q_1,\ldots,Q_j\) so that \eqref{eq:subaccretive} holds for all \(k=1,\ldots,j+1\). Thus the claim follows by induction.
\end{proof}

Let the indexing of the subcubes \(Q_u\) henceforth be the one provided by the Lemma. Let \(\hat{Q}_k:=\bigcup_{u=k}^{2^N}Q_u\), so in particular \(\hat{Q}_1=Q\) and \(\hat{Q}_{2^N}=Q_{2^N}\), and the Lemma implies that
\(\mu(\hat{Q}_k)\gtrsim\mu(Q)\) (while ``\(\leq\)'' is obvious), since \(b\) is bounded.

One obtains the splitting
\begin{equation*}\begin{split}
  \D^b_Q
  &=\Exp^b_{\{Q_1,\ldots,Q_{2^N}\}}-\Exp^b_Q \\
  &=\sum_{u=1}^{2^N-1}[\Exp^b_{\{Q_1,\ldots,Q_u,\hat{Q}_{u+1}\}}-\Exp^b_{\{Q_1,\ldots,Q_{u-1},\hat{Q}_u\}}]
  =:\sum_{u=1}^{2^N-1}\D^b_{Q,u}.
\end{split}\end{equation*}

Now take a closer look at \(\D^b_{Q,u}\); the abbreviation \(f(A):=\int_A f\ud\mu\) will be used, with the same convention for \(b\) in place of \(f\). Assume that \(\mu(Q_u)>0\).
\begin{equation*}\begin{split}
  \D^b_{Q,u}f &=(\Exp^b_{Q_u}+\Exp^b_{\hat{Q}_{u+1}}-\Exp^b_{\hat{Q}_{u}})f \\
   &=b\Big(1_{Q_u}\frac{f(Q_u)}{b(Q_u)}+1_{\hat{Q}_{u+1}}\frac{f(\hat{Q}_{u+1})}{b(\hat{Q}_{u+1})}
        -1_{Q_u\cup\hat{Q}_{u+1}}\frac{f(Q_u)+f(\hat{Q}_{u+1})}{b(Q_u)+b(\hat{Q}_{u+1})}\Big) \\
   &=b\Big(\frac{1_{Q_u}}{b(Q_u)}-\frac{1_{\hat{Q}_{u+1}}}{b(\hat{Q}_{u+1})}\Big)
     \frac{b(Q_u)b(\hat{Q}_{u+1})}{b(\hat{Q}_u)}
     \int\Big(\frac{1_{Q_u}}{b(Q_u)}-\frac{1_{\hat{Q}_{u+1}}}{b(\hat{Q}_{u+1})}\Big)f\ud\mu \\
   &=:b\varphi^b_{Q,u}\int\varphi^b_{Q,u}f\ud\mu,
\end{split}\end{equation*}
where
\begin{equation*}
  \varphi^b_{Q,u}:=
  \sqrt{\frac{b(Q_u)b(\hat{Q}_{u+1})}{b(\hat{Q}_u)}}
  \Big(\frac{1_{Q_u}}{b(Q_u)}-\frac{1_{\hat{Q}_{u+1}}}{b(\hat{Q}_{u+1})}\Big);
\end{equation*}
the choice of the sign of the (in general complex) square root above is irrelevant and may be made arbitrarily.

If \(\mu(Q_u)=0\), then \(\D^b_{Q,u}=0\), and one may define \(\varphi^b_{Q,u}:=0\). The following lemma collects several basic properties of the functions \(\varphi^b_{Q,u}\) which are straightforward consequences of the previous considerations.

\begin{proposition}\label{prop:Haar}
The ``Haar'' functions satisfy
\begin{equation*}
  \int b\varphi^b_{Q,u}\ud\mu=0,
\end{equation*}
and if \(\varphi^b_{Q,u}\not\equiv 0\), then
\begin{equation*}
  \abs{\varphi^b_{Q,u}}\eqsim 
  \sqrt{\mu(Q_u)}\Big(\frac{1_{Q_u}}{\mu(Q_u)}+\frac{1_{\hat{Q}_{u+1}}}{\mu(Q)}\Big).
\end{equation*}
Hence
\begin{equation*}
  \Norm{\varphi^b_{Q,u}}{L^p(\mu)}\eqsim\mu(Q_u)^{1/p-1/2},\qquad p\in[1,\infty],
\end{equation*}
and in particular
\begin{equation*}
  \Norm{\varphi^b_{Q,u}}{L^1(\mu)}\Norm{\varphi^b_{Q,u}}{L^{\infty}(\mu)}\eqsim 1.
\end{equation*}
\end{proposition}


\section{Random dyadic systems; good and bad cubes}\label{sec:random}

In this section I give a convenient parameterization of the dyadic systems as considered above, and use this to introduce a probability distribution on the collection of all such dyadic systems. The construction is equivalent to that used by Nazarov, Treil and Volberg (\cite{NTV:Tb}, Sec.~9.1), but it will be given in a somewhat different and hopefully transparent way.

Let \(\mathscr{D}^0\) denote the standard dyadic system consisting of all \(2^k(n+[0,1[^N)\), where \(k\in\Z\) and \(n\in\Z^N\). A general dyadic system \(\mathscr{D}\) has been defined as a collection \(\mathscr{D}=\bigcup_{k\in\Z}\mathscr{D}_k\) where \(\mathscr{D}_k=x_k+\mathscr{D}^0_k\) for some \(x_k\in\R^N\) and in addition the partition \(\mathscr{D}_k\) refines \(\mathscr{D}_{k+1}\).

There is obviously some redundancy in the choice of \(x_k\), since only its value modulo \(2^k\) (in each coordinate) is relevant. Thus, without loss of generality, it may be assumed that \(x_k\in[0,2^k[^N\). On the other hand, the condition that \(\mathscr{D}_k\) refine \(\mathscr{D}_{k+1}\) can be rephrased as \(x_k\equiv x_{k+1}\mod 2^k\), or in other words \(x_{k+1}=x_k+\beta_k 2^k\) for some \(\beta_k\in\{0,1\}^N\). It follows by iteration that
\begin{equation*}
  x_k=\sum_{j<k}\beta_j 2^j,\qquad\beta_j\in\{0,1\}^N.
\end{equation*}
Hence the whole system \(\mathscr{D}\) can be thought of as a shift of the standard system, \(\mathscr{D}=\mathscr{D}^0+\beta\), where \(\beta\) is the formal power series \(\beta=\sum_{j\in\Z}\beta_j 2^j\), and it is understood that a truncation modulo \(2^k\) of this series is first made before computing the shift \(Q+\beta:=Q+\sum_{j<k}\beta_j 2^j\) for \(Q\in\mathscr{D}^0_k\).

Now that all dyadic systems have been parameterized by \(\beta\in(\{0,1\}^N)^{\Z}\), there is an obvious way to interpret a ``random dyadic system'' by assigning the natural product probability on \((\{0,1\}^N)^{\Z}\) so that the coordinate functions \(\beta_j\) are independent and \(\prob(\beta_{j}=\eta)=2^{-N}\) for all \(\eta\in\{0,1\}^N\). (Actually, since it was required that \(x_k=\sum_{j<k}\beta_j2^{-j}\in[0,2^k[^N\), one should exclude the sequences \(\beta\) with the following property: for some \(k\in\Z\) and \(i\in\{1,\ldots,N\}\), the \(i\)th coordinate of \(\beta_{j}\) equals \(1\) for all \(j<k\). But this does not affect any of the probabilistic statements, since this kind of sequences have probability zero.)

Note that the formal shift parameter \(\beta=\sum_{j\in\Z}\beta_j 2^j\) cannot in general be replaced by real shift by some vector \(x\in\R^N\), and in fact the dyadic systems \(x+\mathscr{D}^0\) have vanishing probability among all dyadic systems. One can show that \(\mathscr{D}=\beta+\mathscr{D}^0\) is of the mentioned special form if and only if there is a \(k\in\Z\) and \(\eta\in\{0,1\}^N\) such that \(\beta_j\equiv \eta\) for all \(j>k\), and clearly this kind of sequences have zero probability among all \((\beta\in\{0,1\}^N)^{\Z}\).

I next recall the notion of singular cubes, essentially following \cite{NTV:Tb}, Def.~7.2. This involves two auxiliary parameters
\begin{equation*}
  \gamma:=\frac{\alpha}{2(\alpha+d)}
\end{equation*}
and (a large) \(r\in\Z_+\), which will be chosen later.  It is required that, at least,
\begin{equation}\label{eq:rAtLeast}
  2^{r(1-\gamma)}\geq 4\lambda,
\end{equation}
where $\lambda$ is the parameter of the $\BMO^p_{\lambda}$ spaces in the assumptions.

A pair of cubes \(\{Q,R\}\) with \(\ell(Q)\leq\ell(R)\) is called singular if
\begin{equation*}
  \dist(Q,\partial R)\leq\ell(Q)^{\gamma}\ell(R)^{1-\gamma}
\end{equation*}
when \(S=R\) or when \(S\) is any one the \(2^N\) dyadic subcubes of \(R\) with \(\ell(S)=\frac{1}{2}\ell(R)\). The pair \(\{Q,R\}\) is called essentially singular if, in addition, \(\ell(Q)\leq 2^{-r}\ell(R)\).

Given two dyadic systems \(\mathscr{D}\) and \(\tilde{\mathscr{D}}\), a cube \(Q\in\mathscr{D}\) is called \emph{bad} with respect to \(\tilde{\mathscr{D}}\) if there exists an \(R\in\tilde{\mathscr{D}}\) with \(\ell(R)\geq\ell(Q)\) such that the pair \(\{Q,R\}\) is essentially singular, and otherwise it is called \emph{good}. Just like in \cite{NTV:Tb}, I will consider martingale difference decompositions with respect to two independent random dyadic systems $\mathscr{D}=\mathscr{D}^{\beta}$ and $\mathscr{D}'=\mathscr{D}^{\beta'}$, where $\beta,\beta'\in(\{0,1\}^N)^{\Z}$. There will be a separate treatment for the good and bad cubes in these systems but, deviating slightly from \cite{NTV:Tb}, the relevant badness of $Q\in\mathscr{D}$ will be defined not directly with respect to $\mathscr{D}'$, but with respect to a modification thereof.

To this end, let us consider two more random binary sequences $\tilde{\beta}$ and $\tilde{\beta}'$. Then I declare that $Q\in\mathscr{D}$ is bad if it is bad with respect to the dyadic system $\mathscr{D}_Q'$ parameterized by the binary sequence $(\tilde{\beta}_j')_{2^j<\ell(Q)}\cup(\beta_j')_{2^j\geq\ell(Q)}$, and similarly $R\in\mathscr{D}'$ is bad if it is bad with respect to the system $\mathscr{D}_R$ parameterized by $(\tilde{\beta}_j)_{2^j<\ell(Q)}\cup(\beta_j)_{2^j\geq\ell(Q)}$. The notation \(\mathscr{D}_{\bad}\), \(\mathscr{D}_{\good}\) will be used for the bad and good subcollections of \(\mathscr{D}\), similarly for \(\mathscr{D}'\).

The advantage of this definition is that the badness or goodness of any $Q\in\mathscr{D}$ is independent of the positions of the relatively smaller cubes $R\in\mathscr{D}'$, $\ell(R)\leq\ell(Q)$, which depend on $(\beta_j')_{2^j<\ell(R)}$ only. This has the following useful consequence:

\begin{lemma}\label{lem:underExpectations}
Let $\phi(Q,R)$ be a function depending on two cubes $Q$ and $R$. With $\beta'$ and $\tilde\beta'$ fixed, under a random choice of $\beta$ and $\tilde\beta$, the expectations of the series,
\begin{equation*}
  \Exp_{\beta}\sum_{R\in\mathscr{D'}}\sum_{\substack{Q\in\mathscr{D} \\ \ell(Q)\leq\ell(R)}}\phi(Q,R)\qquad\text{and}\qquad
  \Exp_{\beta\tilde\beta}\sum_{R\in\mathscr{D}'_{\good}}\sum_{\substack{Q\in\mathscr{D} \\ \ell(Q)\leq\ell(R)}}\phi(Q,R)
\end{equation*}
differ only by a multiplicative factor depending on the parameters $r$ and $\gamma$.
\end{lemma}

Thus, under the expectations, we may take the bigger cube in such summations to be restricted to the good cubes or not, as we wish. Since $\beta'$ and $\tilde\beta'$, which determine the goodness of $Q\in\mathscr{D}$, are kept fixed, the result also remains true if $Q\in\mathscr{D}$ is replaced by $Q\in\mathscr{D}_{\good}$ in both inner sums; this amounts to replacing $\phi(Q,R)$ by $1_{\good}(Q)\phi(Q,R)$, which is just another function of $Q$ and $R$.

\begin{proof}
Observe that, by reasons of symmetry, the probability of goodness
\begin{equation*}
  \pi_{\good}:=\prob_{\beta\tilde\beta}(R\in\mathscr{D}_{\good}')=\Exp_{\beta\tilde\beta} 1_{\good}(R),
\end{equation*}
is independent of the particular $R\in\mathscr{D}_{\good}'$. Using the parameterization in terms of $\mathscr{D}^0$, so that there is no randomness in the summation variable, one may compute
\begin{equation*} 
\begin{split}
  &\pi_{\good}\Exp_{\beta\tilde\beta}\sum_{R\in\mathscr{D'}}\sum_{\substack{Q\in\mathscr{D} \\ \ell(Q)\leq\ell(R)}}\phi(Q,R) \\
  &=\sum_{R\in\mathscr{D}'}\sum_{\substack{Q\in\mathscr{D}^0 \\ \ell(Q)\leq\ell(R)}}
          \Exp_{\beta\tilde\beta} 1_{\good}^{\beta\tilde\beta}(R)
          \Exp_{\beta\tilde\beta}\phi(Q+\beta,R),
\end{split}
\end{equation*}
where the dependence of the goodness of $R$ on $\beta$ and $\tilde\beta$ has been indicated explicitly for clarity.
Here $Q+\beta$, and hence $\phi(Q+\beta,R)$, depends only on $\beta_j$ with $2^j<\ell(Q)$, whereas the goodness of $R$ depends on $\tilde\beta$ and $\beta_j$ for $2^j\geq\ell(R)\geq\ell(Q)$. Thus, by the product rule of expectations of independent quantities, the computation may be continued with
\begin{equation*} 
\begin{split}
  &=\sum_{R\in\mathscr{D}'}\sum_{\substack{Q\in\mathscr{D}^0 \\ \ell(Q)\leq\ell(R)}}
          \Exp_{\beta\tilde\beta} \big(1_{\good}^{\beta\tilde\beta}(R)\phi(Q+\beta,R)\big) \\
   &=\Exp_{\beta\tilde\beta}\sum_{R\in\mathscr{D}'_{\good}}\sum_{\substack{Q\in\mathscr{D} \\ \ell(Q)\leq\ell(R)}}\phi(Q,R),
\end{split}
\end{equation*}
and this is what was claimed.
\end{proof}

The modified notion of good and bad is still reasonably close to the original one, in the following sense:

\begin{lemma}\label{lem:newGood}
If $Q\in\mathscr{D}_{\good}$ and $R\in\mathscr{D}'$ with $\ell(R)\geq 2^r\ell(Q)$, then
\begin{equation*}
  \dist(Q,\partial R)\geq\frac12\ell(Q)^{\gamma}\ell(R)^{1-\gamma}.
\end{equation*}
\end{lemma}

\begin{proof}
Under the assumptions, the cube $\tilde{R}:=R+\sum_{2^j<\ell(Q)}2^j(\tilde{\beta}_j'-\beta_j')$ belongs to $\mathscr{D}'(Q)$ and hence $\dist(Q,\partial\tilde{R})>\ell(Q)^{\gamma}\ell(R)^{1-\gamma}$. But then, using \eqref{eq:rAtLeast},
\begin{equation*}
\begin{split}
  \dist(Q,\partial R) &\geq \dist(Q,\partial\tilde{R})-\ell(Q) \\
   &>\ell(Q)^{\gamma}\ell(R)^{1-\gamma}(1-2^{-(1-\gamma)r})\geq\frac12\ell(Q)^{\gamma}\ell(R)^{1-\gamma}.\qedhere
\end{split}
\end{equation*}
\end{proof}

The main point of considering random dyadic systems is to be able to quantify the sense in which the bad cubes are rare.

\begin{lemma}\label{lem:NTV92}
Let \(\mathscr{D}\) and \(Q\in\mathscr{D}=\mathscr{D}^{\beta}\) be fixed, and choose \(\beta',\tilde{\beta}'\) randomly. Then
\begin{equation*}
  \prob_{\beta',\tilde\beta'}(Q\in\mathscr{D}_{\bad})\leq 2N\frac{2^{-r\gamma}}{1-2^{-\gamma}}.
\end{equation*}
\end{lemma}

Note that the right side can be made smaller than any preassigned \(\epsilon>0\) with a sufficiently large choice of \(r\in\Z_+\).

\begin{proof}
This is \cite{NTV:Tb}, Lemma~9.2. The fact that the badness is defined with respect to $\mathscr{D}'(Q)$ rather than $\mathscr{D}'$ does not change the corresponding probability, since both these dyadic systems are identically distributed.
\end{proof}

Given \(Q\in\mathscr{D}\) and \(n\in\Z_+\), the expression \(Q^{(n)}\) denotes the dyadic ancestor of \(Q\) of the \(n\)th generation, i.e., it is the unique cube such that \(Q\subseteq Q^{(n)}\in\mathscr{D}\) and \(\ell(Q^{(n)})=2^n\ell(Q)\). For indicating the appropriate ancestor in a number of arguments below, it is convenient to introduce the following integer-valued function: for \(j=0,1,2,\ldots\), let
\begin{equation*}
  \theta(j):=\Bceil{\frac{j\gamma+r}{1-\gamma}},
\end{equation*}
where \(\ceil{x}\) is the first (i.e., smallest) integer bigger than or equal to \(x\).


\section{The tangent martingale trick}\label{sec:trick}

In this section I present the central tool for estimating the action on vector-valued random sums of various averaging-type integral operators, which will be encountered in the sequel. While the applications in the present paper will all be in the context of \(\R^N\) and its dyadic cubes, I decided to highlight the abstract nature of this argument by giving the result in a general \(\sigma\)-finite measure space \((E,\mathscr{M},\mu)\) having a refining sequence of partitions as follows: For each \(k\in\Z\), let \(\mathscr{A}_k\) be a countable partition of \(E\) into sets of finite positive measure so that \(\sigma(\mathscr{A}_k)\subseteq\sigma(\mathscr{A}_{k-1})\subseteq\mathscr{M}\), and let \(\mathscr{A}=\bigcup_{k\in\Z}\mathscr{A}_k\).

The basic idea of the tangent martingale trick is the following: Given functions \(f_A\) supported by the atoms \(A\), and of such a form that \(f_A\) is \(\sigma(\mathscr{A}_{k-1})\)-measurable whenever \(A\in\mathscr{A}_k\), these will be replaced by new functions, which have a simpler dependence on the variable \(x\in E\), being just multiples of the indicator \(1_A\), but they still contain all the original information, which is hidden in the dependence on a new variable \(y\).

To make the ``replacement'' precise, for each \(A\in\mathscr{A}\), let \(\nu_A\) denote the probability measure \(\mu(A)^{-1}\cdot\mu|_A\). Let \((F,\mathscr{N},\nu)\) be the space \(\prod_{A\in\mathscr{A}}A\) with the product \(\sigma\)-algebra and measure. Its points will be denoted by \(y=(y_A)_{A\in\mathscr{A}}\). Then the following norm equivalence holds:

\begin{theorem}\label{thm:trick}
If \(X\) is a UMD space and \(p\in(1,\infty)\), then
\begin{equation}\label{eq:trick}\begin{split}
  \iint_{\Omega\times E}
    &\Bnorm{\sum_{k\in\Z}\radem_k\sum_{A\in\mathscr{A}_k}f_A(x)}{X}^p\ud\prob(\radem)\ud\mu(x) \\
  &\eqsim\iiint_{\Omega\times E\times F}
    \Bnorm{\sum_{k\in\Z}\radem_k\sum_{A\in\mathscr{A}_k}1_A(x)f_A(y_A)}{X}^p\ud\prob(\radem)\ud\mu(x)\ud\nu(y)
\end{split}\end{equation}
\end{theorem}

\begin{proof}
This is a version of McConnell's \cite{McConnell} Theorem~2.2 for tangent martingale difference sequences. Consider the space $\Omega\times E\times F$, and identify any function, subset or collection of sets on one of the components with a similar object lifted to the product space in the usual way; so for example $\mathscr{A}_k$ is identified with $\Omega\times\mathscr{A}_k\times F$.
Then consider the functions
\begin{equation*}
  d_k(\radem,x,y):=\radem_k\sum_{A\in\mathscr{A}_k}f_A(x),\quad
  e_k(\radem,x,y):=\radem_k\sum_{A\in\mathscr{A}_k}1_A(x)f_A(y_A),
\end{equation*}
and the \(\sigma\)-algebras \(\mathscr{F}_k:=\sigma(\{\radem_j,y_A:A\in\mathscr{A}_j,j\geq k\},\mathscr{A}_{k-1})\). Then both \(d_k\) and \(e_k\) are \(\mathscr{F}_k\)-measurable and, because of the \(\radem_k\) factor,
\begin{equation*}
  \Exp[d_k|\mathscr{F}_{k+1}]=\Exp[e_k|\mathscr{F}_{k+1}]=0,
\end{equation*}
i.e., they form martingale difference sequences. Moreover, they satisfy the following tangent property: their conditional distributions on \(\mathscr{F}_{k+1}\) coincide, i.e.,
\begin{equation*}
  \Exp[1_{\{d_k\in D\}}|\mathscr{F}_{k+1}]=\Exp[1_{\{e_k\in D\}}|\mathscr{F}_{k+1}]
\end{equation*}
for all Borel sets \(D\subseteq X\). In fact, computing  the conditional expectations is easy in the present case, since this amounts to fixing the variables \(\radem_j\) and \(y_A\), for \(A\in\mathscr{A}_j\) and \(j>k\) (they do not appear in \(d_k\) nor \(e_k\), so this amounts to nothing), and computing the average over \(x\in A\) for every \(A\in\mathscr{A}_k\). But
\begin{equation*}\begin{split}
  &\prob\otimes\mu\otimes\nu(\{d_k\in D\}\cap\{x\in A\})=\prob\otimes\mu(\{\radem_k f_A(x)\in D\}),\\
  &\prob\otimes\mu\otimes\nu(\{e_k\in D\}\cap\{x\in A\})=\prob\otimes\nu(\{\radem_k f_A(y_A)\in D\})\cdot\mu(A),
\end{split}\end{equation*}
which obviously coincide by the definition of \(\nu\).

Hence McConnell's inequality applies. To be precise, he only formulates it in the case of a finite (probability) measure space, rather than a \(\sigma\)-finite one, but one immediately checks that his argument works in the present context as well.
\end{proof}

The main application of the theorem will be via the following consequence, where the auxiliary measure space \(F\) has disappeared:

\begin{corollary}\label{cor:trick}
Let \(X\) be a UMD space and \(p\in(1,\infty)\). For each \(A\in\mathscr{A}\), let \(k_A:A\times A\to\C\) be a jointly measurable function pointwise bounded by \(1\). Then
\begin{equation}\label{eq:trickCor}\begin{split}
  \iint_{\Omega\times E}
    &\Bnorm{\sum_{k\in\Z}\radem_k\sum_{A\in\mathscr{A}_k}
    \frac{1_A(x)}{\mu(A)}\int_A k_A(x,z)f_A(z)\ud\mu(z)}{X}^p\ud\prob(\radem)\ud\mu(x) \\
  &\lesssim  \iint_{\Omega\times E}
    \Bnorm{\sum_{k\in\Z}\radem_k\sum_{A\in\mathscr{A}_k}
    f_A(x)}{X}^p\ud\prob(\radem)\ud\mu(x).
\end{split}\end{equation}
\end{corollary}

\begin{proof}
Consider the uniformly bounded sequence of functions
\begin{equation*}
  K_k(x,y):=\sum_{A\in\mathscr{A}_k}1_A(x)k_A(x,y_A)
\end{equation*}
on \(E\times F\). By the contraction principle, the right (and hence by Theorem~\ref{thm:trick}, the left) side of \eqref{eq:trick} dominates the expression
\begin{equation*}
  \iiint_{\Omega\times E\times F}
    \Bnorm{\sum_{k\in\Z}\radem_k\sum_{A\in\mathscr{A}_k}
    1_A(x)k_A(x,y_A)f_A(y_A)}{X}^p\ud\prob(\radem)\ud\mu(x)\ud\nu(y),
\end{equation*}
which in turn dominates
\begin{equation*}
  \iint_{\Omega\times E}
    \Bnorm{\sum_{k\in\Z}\radem_k\sum_{A\in\mathscr{A}_k} 1_A(x)
     \int_{F}k_A(x,y_A)f_A(y_A)\ud\nu(y)}{X}^p\ud\prob(\radem)\ud\mu(x)
\end{equation*}
by Jensen's inequality. Since the innermost integrand only depends on the coordinate \(y_A\) of \(y\), the integration over \(F\) with respect to \(\ud\nu(y)\) may be replaced by integration over \(A\) with respect to \(\ud\nu_A(y_A)=\mu(A)^{-1}\ud\mu(y_A)\), which completes the proof.
\end{proof}

It might be interesting to note an alternative approach to this result under additional structure on the space \(X\):

\begin{proposition}
Let \(X\) be a UMD function lattice, and \(p\in(1,\infty)\). Then the conclusion of Corollary~\ref{cor:trick} holds even without requiring the \(f_A\) to be \(\sigma(\mathscr{A}_{k-1})\)-measurable for \(A\in\mathscr{A}_k\).
\end{proposition}

\begin{proof}
The main difference compared to a general UMD space is the existence of an absolute value \(\abs{\xi_k}\in X\) for each element \(\xi_k\in X\). This will be exploited via the fact that
\begin{equation*}
  \int_{\Omega}\Bnorm{\sum_{k\in\Z}\radem_k \xi_k}{X}^p\ud\prob(\radem)
  \eqsim\Bnorm{\Big(\sum_{k\in\Z}\abs{\xi_k}^2\Big)^{1/2}}{X}^p
  \eqsim\int_{\Omega}\Bnorm{\sum_{k\in\Z}\radem_k \abs{\xi_k}}{X}^p\ud\prob(\radem).
\end{equation*}

Hence, taking into account the bound \(\abs{k_A(x,z)}\leq 1\),
\begin{equation*}\begin{split}
   LHS\eqref{eq:trickCor}
   &\lesssim\iint_{\Omega\times E}
      \Bnorm{\sum_{k\in\Z}\radem_k\sum_{A\in\mathscr{A}_k}
       \frac{1_A(x)}{\mu(A)}\int_A \abs{f_A(z)}\ud\mu(z)}{X}^p\ud\prob(\radem)\ud\mu(x) \\
   &=\iint_{\Omega\times E}\Bnorm{\sum_{k\in\Z}\radem_k\Exp\big[
       \sum_{A\in\mathscr{A}_k}\abs{f_A}\,\big|\sigma(\mathscr{A}_k)\big](x)}{X}^p
       \ud\prob(\radem)\ud\mu(x) \\
   &\lesssim\iint_{\Omega\times E}\Bnorm{\sum_{k\in\Z}\radem_k
       \sum_{A\in\mathscr{A}_k}\abs{f_A(x)}}{X}^p
       \ud\prob(\radem)\ud\mu(x) \eqsim RHS\eqref{eq:trickCor},
\end{split}\end{equation*}
where the second to last step was an application of the vector-valued Stein inequality due to Bourgain~\cite{Bourgain:86}, and throughout it was used that the functions $f_A$, where $A\in\mathscr{A}_k$ for a fixed $k$, are disjointly supported.
\end{proof}


\section{Separated cubes}\label{sec:separated}

This section begins the lenghty task of estimating various subseries of the expansion~\eqref{eq:matrixT}. The numbers $\epsilon$ and $m$ will be considered fixed, and the summation ranges of dyadic cubes are restricted to the side-lengths $\epsilon\leq\ell(Q),\ell(R)\leq 2^m$ without indicating this explicitly; it is important that the obtained estimates are uniform in these parameters, and the convention concerning implicit constants, as formulated in Notation~\ref{not:C} will be heavily employed. In addition to the parameters listed there, the implicit constants are also allowed to depend on the auxiliary parameter $r\in\Z_+$ from the definition of good and bad dyadic cubes in Section~\ref{sec:random}. This convention will be in force until further notice.

In this section, the binary sequences $\beta,\beta',\tilde\beta,\tilde\beta'$ parameterizing the dyadic systems are held fixed, and one deals with the part of~\eqref{eq:matrixT} where a smaller cube \(Q\in\mathscr{D}\) is separated from the larger \(R\in\mathscr{D}'\) by at least its own side-lenght, \(\dist(Q,R)\geq\ell(Q)\). By symmetry of the assumptions, the same conclusion will follow for the part of the series with the r\^oles of \(Q\) and \(R\) interchanged. It will also be assumed that all ``Haar'' functions $\varphi_Q$ related to the smaller cube $Q$ are cancellative ones; in any case, the contrary could only happen when $\ell(Q)=\ell(R)=2^m$, and the boundedly many pairs of cubes like this will be treated as close-by cubes of comparable size in Section~\ref{sec:close}.

Also, one restricts the summations to the good cubes $Q\in\mathscr{D}$ and $R\in\mathscr{D}'$ only. This is equivalent to replacing $f$ and $g$ by their good parts
\begin{equation*}
  f_{\good}:=\sum_{\substack{Q\in\mathscr{D}_{\good} \\ \ell(Q)\leq 2^m}}\D_Q^{b_1}f+\sum_{\substack{Q\in\mathscr{D}_{\good} \\ \ell(Q)=2^m}}\Exp_Q^{b_1}f
\end{equation*}
and $g_{\good}$, which is defined similarly. The present aim is then to prove that
\begin{equation}\label{eq:separGoal}
  \Babs{\sum_{R\in\mathscr{D}_{\good}'}
    \sum_{\begin{smallmatrix}Q\in\mathscr{D}_{\good}\\ \ell(Q)\leq\dist(Q,R)\wedge\ell(R)\end{smallmatrix}}
    \pair{g}{\psi_R}T_{RQ}\pair{\varphi_Q}{f}}
  \lesssim\Norm{g}{L^{p'}(\mu;X^*)}\Norm{f}{L^p(\mu;X)}
\end{equation}
where, recall,
\(T_{RQ}:=\pair{\psi_R b_2}{T(b_1\varphi_Q)}\).

\begin{lemma}\label{lem:NTV61}
Let \(\ell(Q)\leq\ell(R)\wedge\dist(Q,R)\). Then
\begin{equation*}
  \abs{T_{RQ}}\lesssim\frac{\ell(Q)^{\alpha}}{\dist(Q,R)^{d+\alpha}}
    \Norm{\psi_R}{L^1(\mu)}\Norm{\varphi_Q}{L^1(\mu)}.
\end{equation*}
\end{lemma}

\begin{proof}
This is essentially \cite{NTV:Tb}, Lemma~6.1
(but leaving out the last line of the proof, where the \(\Norm{\varphi_Q}{L^1(\mu)}\) was dominated by \(\mu(Q)^{1/2}\Norm{\varphi_Q}{L^2(\mu)}\), and similarly with \(\psi_Q\)):
writing \(y_0\) for the centre of \(Q\) and recalling that \(\int b_1\varphi_Q\ud\mu=0\), there holds
\begin{equation*}
  \abs{T_{RQ}}=\Babs{\iint\psi_R(x)b_2(x)[K(x,y)-K(x,y_0)]b_1(y)\varphi_Q(y)\ud\mu(y)\ud\mu(x)},
\end{equation*}
and the required estimate follows from \eqref{eq:CZK2}.
\end{proof}

For the next estimate, define, as in \cite{NTV:Tb}, the \emph{long distance} of two cubes
\begin{equation*}
  D(Q,R):=\ell(Q)+\dist(Q,R)+\ell(R).
\end{equation*}

\begin{lemma}\label{lem:NTV64}
Let \(Q\in\mathscr{D}_{\good}\) and \(R\in\mathscr{D}'\) be as in Lemma~\ref{lem:NTV61}. Then
\begin{equation*}
  \abs{T_{RQ}}\lesssim\frac{\ell(Q)^{\alpha/2}\ell(R)^{\alpha/2}}{D(Q,R)^{d+\alpha}}
    \Norm{\psi_R}{L^1(\mu)}\Norm{\varphi_Q}{L^1(\mu)}.
\end{equation*}
\end{lemma}

\begin{proof}
This repeats \cite{NTV:Tb}, Lemma~6.4.
\end{proof}

To prove \eqref{eq:separGoal}, consider first the part of the series where the ratio \(\ell(R)/\ell(Q)\) is a fixed number \(2^n\) with \(n\in\N\), and also \(2^j<D(Q,R)/\ell(R)\leq 2^{j+1}\) for a momentarily fixed \(j\in\N\). The last double inequality will be abbreviated as \(D(Q,R)/\ell(R)\sim 2^j\). If moreover \(R\in\mathscr{D}_k'\), the estimate of Lemma~\ref{lem:NTV64} reads
\begin{equation}\label{eq:NTV64}
  \frac{\abs{T_{RQ}}}{\Norm{\psi_R}{1}\Norm{\varphi_Q}{1}}
  \lesssim\frac{2^{(k-n)\alpha/2}2^{k\alpha/2}}{2^{(k+j)(d+\alpha)}}
  =2^{-n\alpha/2}2^{-j\alpha}2^{-(k+j)d}.
\end{equation}
In the following calculations, the summation condition \(\dist(Q,R)\geq\ell(Q)\) is always in force although it will not be indicated explicitly.

From~\eqref{eq:firstBy}, it follows that
\begin{equation}\label{eq:separStarts}\begin{split}
  &\Babs{\sum_{k\in\Z}\sum_{R\in\mathscr{D}_{\good,k}'}
    \sum_{\ontop{Q\in\mathscr{D}^{\good}_{k-n}}{D(Q,R)/\ell(R)\sim 2^j}}
    \pair{g}{\psi_R}T_{RQ}\pair{\varphi_Q}{f}} \\
  &\lesssim\Norm{g}{L^{p'}(\prob\otimes\mu;X^*)}
   \BNorm{\sum_{k\in\Z}\radem_k\sum_{\ontop{R\in\mathscr{D}_{\good,k}';Q\in\mathscr{D}^{\good}_{k-n}}{{D(Q,R)}/{\ell(R)}\sim 2^j}}
     \psi_R\, T_{RQ}\pair{\varphi_Q}{f}}{L^p(\prob\otimes\mu;X)}.
\end{split}\end{equation}

Next, observe that all cubes \(R\in\mathscr{D}_{\good}'\) with \(\ell(R)=2^{n}\ell(Q)\) and \(D(R,Q)\leq 2^{j+1}\ell(R)\) satisfy
\begin{equation*}
  R\subseteq Q^{(n+j+\theta(j))}.
\end{equation*}
Indeed, if not, then a contradiction results from \eqref{eq:rAtLeast} and
\begin{equation*}\begin{split}
  2^{j+1}\ell(R)
  &\geq D(R,Q)>\dist(R,Q)
   \geq\dist(R,Q^{(n+j+\theta(j))}) \\
  &\geq\frac12\ell(R)^{\gamma}\ell(Q^{(n+j+\theta(j))})^{1-\gamma}
   =\frac12\ell(R)^{\gamma}(2^{j+\theta(j)}\ell(R))^{1-\gamma} \\
  &\geq 2^{-1}2^{j(1-\gamma)+\gamma j+r}\ell(R)
   =2^{j+r-1}\ell(R).
\end{split}\end{equation*}
Hence the summation over \(R\) may be reorganized as
\begin{equation*}
  \sum_{R\in\mathscr{D}_{\good,k}'}
  =\sum_{S\in\mathscr{D}_{k+j+\theta(j)}}
   \sum_{\ontop{R\in\mathscr{D}_{\good,k}'}{R\subset S}}.
\end{equation*}
For \(Q,R,S\) as in the above sums, denote
\begin{equation*}\begin{split}
  T_{RQ} &=:2^{-n\alpha/2}2^{-j\alpha}\frac{\Norm{\psi_R}{1}\Norm{\varphi_Q}{1}}{2^{(k+j)d}}t_{RQ} \\
  &=:2^{-(n+j)\alpha/2}\frac{\Norm{\psi_R}{1}\Norm{\varphi_Q}{1}}{\mu(S)}
    \tilde{t}_{RQ},
\end{split}\end{equation*}
where \(\abs{\tilde{t}_{RQ}}\lesssim\abs{t_{RQ}}\lesssim 1\) by \(\mu(S)\leq 2^{(k+j+\theta(j))d}\lesssim 2^{(k+j)d+j\alpha/2}\) and \eqref{eq:NTV64}.

For each \(S\in\mathscr{D}_{k+j+\theta(j)}\), define the kernel
\begin{equation*}
  K_S(x,y):=\sum_{\ontop{R\in\mathscr{D}_{\good,k}'}{R\subset S}}
   \sum_{\ontop{Q\in\mathscr{D}^{\good}_{k-n}}{D(Q,R)/\ell(R)\sim 2^j}}
   \psi_R(x)\Norm{\psi_R}{1}\tilde{t}_{RQ}\Norm{\varphi_Q}{1}\varphi_Q(y)b_1(y).
\end{equation*}
Then \(K_S\) is supported on \(S\times S\) and \(\abs{K_S(x,y)}\lesssim 1\), since \(\Norm{\varphi_Q}{\infty}\Norm{\varphi_Q}{1}\lesssim 1\), and the same with \(\psi_R\), and since there is at most one non-zero term in the double sum for any given pair of points \((x,y)\). The quantity inside the \(L^p(\prob\otimes\mu;X)\)-norm in \eqref{eq:separStarts} is \(2^{-(n+j)\alpha/2}\) times
\begin{equation*}
  \sum_{k_0=0}^{n+j+\theta(j)}
  \sum_{\ontop{k\in\Z;k\equiv k_0}{\mod n+j+\theta(j)+1}}
  \radem_k\sum_{S\in\mathscr{D}_{k+j+\theta(j)}}
  \frac{1_S(x)}{\mu(S)}\int_S K_S(x,y)\frac{1_S\D_{k-n}^{b_1}f}{b_1}(y)\ud y,
\end{equation*}
where the fact that \(\pair{\varphi_Q}{f}=\pair{\varphi_Q}{\D_{k-n}^{b_1}f}\) for \(Q\in\mathscr{D}_{k-n}\) was also used.

For a fixed \(k_0\), the series over \(k\equiv k_0\mod n+j+\theta(j)+1\) above is exactly of the form considered in Corollary~\ref{cor:trick}: \(1_S\cdot b_1^{-1}\cdot\D_k^{b_1}f\) is supported on \(S\in\mathscr{D}_{k+j+\theta(j)}\), and it is constant on every cube \(Q'\in\mathscr{D}_{k-n-1}=\mathscr{D}_{k'+j+\theta(j)}\), where \(k'=k-(n+j+\theta(j)+1)\). Hence the \(L^p(\prob\otimes\mu;X)\)-norm of this series is dominated by
\begin{equation*}
  \BNorm{\sum_{k\equiv k_0}\radem_k\sum_{S\in\mathscr{D}_{k+j+\theta(j)}}
     1_S\cdot b_1^{-1}\cdot\D_k^{b_1}f}{L^p(\prob\otimes\mu;X)}
  \lesssim\Norm{f}{L^p(\mu;X)}
\end{equation*}
using Corollary~\ref{cor:trick}, para-accretivity of \(b_1\), and the unconditional convergence of the twisted martingale differences.

The full series over \(k\in\Z\) consists of \(n+j+\theta(j)+1\lesssim n+j+1\) subseries like this, which implies that the quantity in \eqref{eq:separStarts} is dominated by
\begin{equation*}
  C2^{-(n+j)\alpha/2}(n+j+1).
\end{equation*}
Since this is summable over \(n,j\in\N\), this proves the goal~\eqref{eq:separGoal}.

 
\section{Cubes well inside another cube}\label{sec:contained}

This section addresses the part of the series \eqref{eq:matrixT}, where a smaller cube \(Q\in\mathscr{D}_{\good}\) is contained in a substantially larger cube \(R\in\mathscr{D}_{\good}'\) with \(\ell(R)>2^r\ell(Q)\). (Again, the symmetry of the assumptions allows to deduce the same final result also for \(Q\) and \(R\) in opposite relative positions.) Hence, the relevant part of the series of is
\begin{equation}\label{eq:containedSeries}
  \sum_{R\in\mathscr{D}_{\good}'}
    \sum_{\ontop{Q\in\mathscr{D}_{\good};Q\subset R}{\ell(Q)<2^{-r}\ell(R)}}
    \pair{g}{\psi_R}T_{RQ}\pair{\varphi_Q}{f}
\end{equation}
where, as before, \(T_{RQ}:=\pair{\psi_R b_2}{T(b_1\varphi_Q)}\). In this section, we deal with a modification of this series, with $T_{RQ}$ replaced by
\begin{equation*}
  \tilde{T}_{RQ}:=T_{RQ}-\pair{b_2}{T(b_1\varphi_Q)}\ave{\psi_R}_{Q},
\end{equation*}
postponing the treatment of the correction term.
Hence the goal of this section is reduced to proving that
\begin{equation}\label{eq:containedGoal}
  \Babs{\sum_{R\in\mathscr{D}_{\good}'}
    \sum_{\begin{smallmatrix}Q\in\mathscr{D}_{\good};Q\subset R \\ \ell(Q)<2^{-r}\ell(R)\end{smallmatrix}}
    \pair{g}{\psi_R}\tilde{T}_{RQ}\pair{\varphi_Q}{f}}
  \lesssim\Norm{g}{L^{p'}(\mu;X^*)}\Norm{f}{L^p(\mu;X)}.
\end{equation}
This will follow a similar strategy as in Section~\ref{sec:separated}, starting from the estimation of the matrix elements \(\tilde{T}_{RQ}\).

\begin{lemma}\label{lem:NTV73}
Let \(Q\) be good, \(Q\subset R\), \(\ell(Q)<2^{-r}\ell(R)\), and 
let \(S\in\mathscr{D}'\) be the subcube of \(R\) containing \(Q\) with \(\ell(S)=\ell(R)/2\). Then
\begin{equation*}\begin{split}
  \tilde{T}_{RQ} 
  &=-\ave{\psi_R}_{S}\pair{1_{S^c}b_2}{T(b_1\varphi_Q)}
   +\sum_{\ontop{S'\in\mathscr{D}';S'\subset R\setminus S}{\ell(S')=\ell(R)/2}}
      \pair{\psi_R 1_{S'}b_2}{T(b_1\varphi_Q)} \\
  \abs{\tilde{T}_{RQ}}
  &\lesssim\Big(\frac{\ell(Q)}{\ell(R)}\Big)^{\alpha/2}\Big(
    \abs{\ave{\psi_R}_{S}}+\frac{\Norm{\psi_R}{L^1(\mu)}}{\mu(R)}\Big)
   \Norm{\varphi_Q}{L^1(\mu)}
\end{split}\end{equation*}
\end{lemma}

\begin{proof}
Concerning the equality, the fact that \(\psi_R\) is constant on the subcubes of \(R\) implies that
\begin{equation*}
  LHS=\sum_{\ontop{S'\in\mathscr{D}';S'\subset R}{\ell(S')=\ell(R)/2}}
     \pair{\psi_R 1_{S'}b_2}{T(b_1\varphi_Q)}
     -\ave{\psi_R}_{S}\pair{b_2}{T(b_1\varphi_Q)}
  =RHS.
\end{equation*}

As for the upper bound of the first term,
\begin{equation*}\begin{split}
  \abs{\pair{1_{S^c}b_2}{T(b_1\varphi_Q)}}
  &\lesssim\int_{S^c}\frac{\ell(Q)^{\alpha}}{\dist(x,Q)^{d+\alpha}}\Norm{\varphi_Q}{1}\ud\mu(x) \\
  &\lesssim\frac{\ell(Q)^{\alpha}}{\dist(S^c,Q)^{\alpha}}\Norm{\varphi_Q}{1}:
\end{split}\end{equation*}
the first estimate is similar to Lemma~\ref{lem:NTV61}, and the second follows by splitting the integration into dyadic annuli \(2^k\leq\dist(x,Q)/\dist(S^c,Q)<2^{k+1}\), \(k\in\N\), and using
\begin{equation*}
   \mu\big(\{x:\dist(x,Q)<2^k\dist(S^c,Q)\}\big)\lesssim(\ell(Q)+2^k\dist(S^c,Q))^d
   \lesssim 2^{kd}\dist(S^c,Q)^d.
\end{equation*}
The last bound was due to the goodness of \(Q\), and for the same reason (and noting that $\gamma\leq\frac12$)
\begin{equation*}
  \dist(S^c,Q)\gtrsim\ell(Q)^{\gamma}\ell(S)^{1-\gamma}\gtrsim\ell(Q)^{1/2}\ell(R)^{1/2},
\end{equation*}
which concludes the estimation of the first term.

For the second term one can apply Lemma~\ref{lem:NTV64} with \(\psi_R 1_{S'}\) in place of \(\psi_R\), observing that nothing but the support and integrability properties of \(\psi_R\) were used in the proof. This gives
\begin{equation*}\begin{split}
  \abs{\pair{\psi_R 1_{S'}b_2}{T(b_1\varphi_Q)}}
  &\lesssim\frac{\ell(Q)^{\alpha/2}\ell(S')^{\alpha/2}}{D(Q,S')^{d+\alpha}}\Norm{\psi_R}{1}\Norm{\varphi_Q}{1} \\
  &\leq\Big(\frac{\ell(Q)}{\ell(S')}\Big)^{\alpha/2}\frac{\Norm{\psi_R}{1}\Norm{\varphi_Q}{1}}{\ell(S')^d},
\end{split}\end{equation*}
and the proof is concluded by noting that \(\ell(S')=\ell(R)/2\), and hence \(\ell(S')^d\gtrsim\mu(R)\).
\end{proof}

\begin{lemma}\label{lem:containedMatrix}
Under the assumptions of Lemma~\ref{lem:NTV73},
\begin{equation*}
  \abs{\psi_R(x)\tilde{T}_{RQ}\varphi_Q(y)}\lesssim \Big(\frac{\ell(Q)}{\ell(R)}\Big)^{\alpha/2}
   \Big(\frac{1_{R\setminus S}(x)}{\mu(R)}+\frac{1_{S}(x)}{\mu(S)}\Big).
\end{equation*}
\end{lemma}

\begin{proof}
For the second term in the estimate of \(\abs{\tilde{T}_{RQ}}\) in Lemma~\ref{lem:NTV73}, this is clear. For the first term, one has to look more carefully into the structure of the function \(\psi_R\), recalling that \(\psi_R=\varphi^{b_2}_{R,v}\) for some \(v\in\{1,\ldots,2^N\}\). Let \(S=R_w\).

If \(v=w\), then for \(x\in R_w\),
\begin{equation*}
  \abs{\psi_R(x)}=
  \abs{\ave{\psi_R}_{R_w}}\lesssim\mu(R_w)^{-1/2}
\end{equation*}
so that \(\abs{\psi_R(x)\ave{\psi_R}_{R_w}}\lesssim\mu(R_w)^{-1}\), whereas for \(x\in R\setminus R_w\),
\begin{equation*}
  \abs{\psi_R(x)}\lesssim\frac{\mu(R_w)^{1/2}}{\mu(R)},\qquad
  \abs{\psi_R(x)\ave{\psi_R}_{R_w}}\lesssim \frac{1}{\mu(R)}.
\end{equation*}

If \(v\neq w\), then for all \(x\in R\)
\begin{equation*}
  \abs{\ave{\psi_R}_{R_w}}\cdot\Norm{\psi_R}{\infty}
  \lesssim\frac{\mu(R_v)^{1/2}}{\mu(R)}\cdot\frac{1}{\mu(R_v)^{1/2}}
  =\frac{1}{\mu(R)},
\end{equation*}
which is even slightly better than the worst case scenario \(v=w\).
\end{proof}

To prove~\eqref{eq:containedGoal}, consider the part of the sum with \(w\in\{1,\ldots,2^N\}\) fixed and \(Q\subset R_w\). Let further \(n\in\{r+1,r+2,\ldots\}\) be fixed, and \(\ell(Q)=2^{-n}\ell(R)\).

By~\eqref{eq:firstBy}, one gets
\begin{equation}\label{eq:containedStarts}\begin{split}
  &\Babs{\sum_{k\in\Z}\sum_{R\in\mathscr{D}_k'}
    \sum_{\ontop{Q\in\mathscr{D}^{\good}_{k-n}}{Q\subset R_w}}
    \pair{g}{\psi_R}\tilde{T}_{RQ}\pair{\varphi_Q}{f}} \\
  &\lesssim\Norm{g}{}
   \BNorm{\sum_{k\in\Z}\radem_k\sum_{R\in\mathscr{D}_k'}
     \sum_{\ontop{Q\in\mathscr{D}^{\good}_{k-n}}{Q\subset R_w}}
     \psi_R\, \tilde{T}_{RQ}\pair{\varphi_Q}{f}}{L^p(\prob\otimes\mu;X)}.
\end{split}\end{equation}

For each \(k\in\Z\) and \(R\in\mathscr{D}_{k,\good}'\), define the kernels
\begin{equation*}\begin{split}
  K_R^{\operatorname{out}}(x,y) &:=2^{n\alpha/2}
   \sum_{\ontop{Q\in\mathscr{D}^{\good}_{k-n}}{Q\subset R_w}}
    \mu(R)1_{R\setminus R_w}(x)\psi_R(x)
    \tilde{T}_{RQ}\varphi_Q(y)b_1(y), \\
  K_R^{\operatorname{in}}(x,y) &:=2^{n\alpha/2}
    \sum_{\ontop{Q\in\mathscr{D}^{\good}_{k-n}}{Q\subset R_w}}
    \mu(R_w)1_{R_w}(x)\psi_R(x)
    \tilde{T}_{RQ}\varphi_Q(y)b_1(y).
\end{split}\end{equation*}
Then \(K_R^{\operatorname{out}}\) is supported in \(R\times R\) and \(K_{R}^{\operatorname{in}}\) in \(R_w\times R_w\), and they satisfy
\begin{equation*}
  \Norm{K_R^{\operatorname{out}}}{\infty}+\Norm{K_R^{\operatorname{in}}}{\infty}\lesssim 1
\end{equation*}
by Lemma~\ref{lem:containedMatrix}. Moreover,
\begin{equation}\label{eq:containedExpand}\begin{split}
  &\sum_{k\in\Z}\radem_k\sum_{R\in\mathscr{D}_{k,\good}'}
     \sum_{\begin{smallmatrix}Q\in\mathscr{D}^{\good}_{k-n}\\ Q\subset R_w\end{smallmatrix}}
     \psi_R(x) \tilde{T}_{RQ}\pair{\varphi_Q}{f} \\
  &=2^{-n\alpha/2}\sum_{k\in\Z}\radem_k\sum_{R\in\mathscr{D}_{k,\good}'}
    \frac{1_R(x)}{\mu(R)}\int_R K_R^{\operatorname{out}}(x,y)
    \frac{1_R\D_{k-n}^{b_1}f}{b_1}(y)\ud\mu(y) \\
  &\quad+2^{-n\alpha/2}\sum_{k\in\Z}\radem_k\sum_{R\in\mathscr{D}_{k,\good}'}
    \frac{1_{R_w}(x)}{\mu(R_w)}\int_{R_w} K_R^{\operatorname{in}}(x,y)\frac{1_{R_w}\D_{k-n}^{b_1}f}{b_1}(y)\ud\mu(y). 
\end{split}\end{equation}

The aim is again to apply Corollary~\ref{cor:trick}. However, there is an initial obstruction (which I had originally overlooked, and which was kindly pointed to me by Antti V\"ah\"akangas): the kernel functions \(K_R^{\operatorname{out}}\) and \(K_R^{\operatorname{in}}\) are adapted to the dyadic cubes \(R\in\mathscr{D}_k'\) and \(R_w\in\mathscr{D}_{k-1}'\), whereas $\D_{k-n}^{b_1}f/b_1$ is constant on the cubes \(Q\in\mathscr{D}_{k-n-1}\), and this dyadic system is not a refinement of the other one. The way out is to define somewhat finer partitions of \(\R^N\) by
\begin{equation*}
  \mathscr{E}_k':=\{S\cap Q\neq\varnothing:S\in\mathscr{D}_k',Q\in\mathscr{D}_{k+r+1}\},
\end{equation*}
so that the generated \(\sigma\)-algebras \(\sigma(\mathscr{E}_k')=\sigma(\mathscr{D}_k',\mathscr{D}_{k+r+1})\) again form a filtration.

The important point is to observe that
\begin{equation}\label{eq:goodPreserved}
  \mathscr{D}_{k,\good}'\subset\mathscr{E}_k',\quad
  \{R_w\in\mathscr{D}_{k-1}':R_w\subset R\in\mathscr{D}_{k,\good}'\}\subset\mathscr{E}_{k-1}'.
\end{equation}
Indeed, \(R\in\mathscr{D}_{k,\good}'\) means in particular that \(R\) does not intersect the boundary of any \(Q\in\mathscr{D}_{k+r}\), and hence also \(R_w\subset R\) cannot intersect the boundary of any \(Q\in\mathscr{D}_{(k-1)+(r+1)}\). On the right of \eqref{eq:containedExpand}, one may hence replace the first summation condition \(R\in\mathscr{D}_{k,\good}'\) by \(R\in\mathscr{E}_k'\), simply setting \(K_R^{\operatorname{out}}:=0\) for all the new sets \(R\) thus introduced. The second summation in \eqref{eq:containedExpand} could be equally well taken with respect to the summation variable \(R_w\), and \eqref{eq:goodPreserved} allows to write the summation condition as \(R_w\in\mathscr{E}_{k-1}'\), again defining the newly introduced kernels \(K_R^{\operatorname{in}}:=0\). Since $\D_{k-n}^{b_1}f/b_1$ is constant on the sets $Q\in\mathscr{D}_{k-n-1}$, it is \emph{a fortiori} constant on the smaller sets $S\cap Q\in\mathscr{E}_{k-n-r-2}'$.

Splitting the \(k\)-series in \eqref{eq:containedExpand} into \(n+r+2\) subseries according to \(k\equiv k_0\mod n+r+2\), everything is now ready for the application of Corollary~\ref{cor:trick}, just as in Section~\ref{sec:separated}. This provides the upper bound \(C2^{-n\alpha/2}(n+r+2)\) for the \(L^p(\prob\otimes\mu;X)\)-norm in \eqref{eq:containedStarts}, and it is possible to sum over \(n\in\{r+1,r+2,\ldots\}\) and \(w\in\{1,\ldots,2^N\}\) to conclude \eqref{eq:containedSeries}.


\section{The correction term as a paraproduct}\label{sec:para}

It is time to take up the consideration of the correction term subtracted at the beginning of the previous section, namely
\begin{equation}\label{eq:correctionTerm}
\begin{split}
  &\sum_{R\in\mathscr{D}_{\good}'}\sum_{\substack{Q\in\mathscr{D}_{\good};Q\subset R \\ \ell(Q)<2^{-r}\ell(R)}}\sum_{u,v}
    \pair{g}{\psi_{R,v}}\ave{\psi_{R,v}}_Q\pair{b_2}{T(b_1\varphi_{Q,u})}\pair{\varphi_{Q,u}}{f} \\
   &=\sum_{Q\in\mathscr{D}_{\good}}\Big(\sum_{\substack{R\in\mathscr{D}_{\good}',R\supset Q \\ 2^r\ell(Q)<\ell(R)\leq 2^m}}\ave{\D_R^{b_2}g/b_2}_Q \\
   &\qquad\qquad\qquad   +\sum_{\substack{R\in\mathscr{D}_{\good}',R\supset Q \\ 2^r\ell(Q)<\ell(R)=2^m}}\ave{\Exp_R^{b_2}g/b_2}_Q\Big)
          \pair{T^*b_2}{\D_Q^{b_1}f},
\end{split}
\end{equation}
where the previously suppressed variables $u,v$, as well as the implicit restriction to side-lengths at most $2^m$, have been momentarily taken back into consideration.
Were it not for the restriction to the good cubes $R$ only, recalling the definition of $\D_k^{b_2}$ as a difference of two $\Exp_k^{b_2}$'s, the inner summation on the right would be a telescopic one, collapsing to
\begin{equation*}
  \ave{\Exp_{R}^{b_2}g/b_2}_{Q}=\ave{g}_{R}/\ave{b_2}_{R},
\end{equation*}
where $R\in\mathscr{D}'$ is the unique cube such that $\ell(R)=2^r\ell(Q)$ and $R\supset Q$.

It is here that, for the first time, the random choice of the dyadic systems comes to rescue. The double sum on the right of \eqref{eq:correctionTerm} is of the form considered in Lemma~\ref{lem:underExpectations}, which states that, on average, the restriction to good cubes is irrelevant as far as the bigger cubes in such a sum are concerned. More precisely, using the observation concerning the telescopic series, it follows that
\begin{equation*}
  \Exp_{\beta\tilde\beta}\eqref{eq:correctionTerm}
  =\pi_{\good}\Exp_{\beta}\sum_{Q\in\mathscr{D}_{\good}}\sum_{\substack{R\in\mathscr{D}' ,R\supset Q\\ \ell(R)=2^r\ell(Q)}}
      \frac{\ave{g}_R}{\ave{b_2}_R}\pair{T^*b_2}{\D_Q^{b_1}f},
\end{equation*}
where the result was intensionally written as a double sum, even though the cube $R$ is uniquely determined by $Q$. This was done in order to realize that Lemma~\ref{lem:underExpectations} may be applied again, leading back to the restriction into good $R$ only, but now in the collapsed series obtained:
\begin{equation*}
\begin{split}
   \Exp_{\beta\tilde\beta}\eqref{eq:correctionTerm}
   &=\Exp_{\beta\tilde\beta}\sum_{Q\in\mathscr{D}_{\good}}\sum_{\substack{R\in\mathscr{D}_{\good}' ,R\supset Q\\ \ell(R)=2^r\ell(Q)}}
      \frac{\ave{g}_R}{\ave{b_2}_R}\pair{T^*b_2}{\D_Q^{b_1}f} \\
    &=\Exp_{\beta\tilde\beta}\sum_{R\in\mathscr{D}_{\good}'}\sum_{\substack{Q\in\mathscr{D}_{\good} ,Q\subset R \\ \ell(Q)=2^{-r}\ell(R)}}\sum_u
      \frac{\ave{g}_R}{\ave{b_2}_R}\pair{T^*b_2}{b_1\varphi_{Q,u}}\pair{\varphi_{Q,u}}{f}.
\end{split}
\end{equation*}

Now that this computation is done, the summation over $u$ is once again suppressed, one forgets about the average $\Exp_{\beta\tilde\beta}$, and turns into considering the above expression for an arbitrary but fixed choice of the $\beta$ parameters. The above sum will be interpreted as a pairing $\pair{\Pi_2 g}{f}$, where
\begin{equation}\label{eq:Pi2}
  \Pi_2 g:=\sum_{R\in\mathscr{D}_{\good}'}\sum_{\substack{Q\in\mathscr{D}_{\good} ,Q\subset R \\ \ell(Q)=2^{-r}\ell(R)}}
      \frac{\ave{g}_R}{\ave{b_2}_R}\pair{T^*b_2}{b_1\varphi_Q}\cdot\varphi_Q
\end{equation}
is the \emph{paraproduct}, similar to the one in \cite{NTV:Tb}, Section~7.1, but also different in certain respects. Nazarov et al.\ allowed somewhat more summands on the right by imposing only the condition \(\dist(Q,R^c)\geq\lambda\ell(Q)\) for the two cubes $Q\in\mathscr{D}$ and $R\in\mathscr{D}'$ with side-lengths related as above. This is a consequence of \(Q\in\mathscr{D}_{\good}\), which implies that $\dist(Q,\partial R)\geq\frac12\ell(Q)^{\gamma}\ell(R)^{1-\gamma}\geq\frac12\cdot 4\lambda\ell(Q)$, and the difference between these two conditions is inessential for the present considerations.

More important, however, is the fact that one can also restrict the bigger cube $R$ to be a good one here. This condition was absent from the paraproduct treated in earlier versions of this paper, which I only managed to do with the help of Therem~\ref{thm:PfRMF} and the additional assumption of the RMF property of the Banach space $X^*$. I still do not know whether the variant of the paraproduct with $R\in\mathscr{D}'$  could be treated without the RMF condition, but the reduction above shows that one only needs to consider $\Pi_2$ as in \eqref{eq:Pi2}.

This paraproduct is related to, and should be controlled in terms of the BMO function \(T^* b_2\). The membership in BMO will be exploited via the following estimate:

\begin{lemma}\label{lem:NTV51}
For \(p\in(1,\infty)\) and \(h\in\BMO_{\lambda}^p(\mu)\),
\begin{equation*}
  \BNorm{
    \sum_{\ontop{Q\in\mathscr{D}_{\good};Q\subset R}{\ell(Q)\leq 2^{-r}\ell(R)}}
    \radem_Q\pair{h}{b_1\varphi_Q}\varphi_Q}{L^p(\prob\otimes\mu)}
  \lesssim \mu(R)^{1/p}\Norm{h}{\BMO_{\lambda}^p(\mu)}.
\end{equation*}
\end{lemma}

\begin{proof}
Consider the Whitney-type covering \(\mathscr{W}\) of \(R\) consisting of the maximal dyadic cubes \(S\in\mathscr{D}\) subject to the conditions \(\ell(S)\leq 2^{-r}\ell(R)\) and
\begin{equation}\label{eq:separation}
  \dist(S,R^c)\geq\lambda\ell(S).
\end{equation}
Then the expanded cubes \(\lambda S\) satisfy the bounded overlapping property
\begin{equation*}
  \sum_{S\in\mathscr{W}}1_{\lambda S}\leq C 1_R.
\end{equation*}

If \(Q\) is one of the cubes appearing in the sum on the left of the assertion, the goodness of \(Q\) implies that
\begin{equation*}
  \dist(Q,R^c)\geq\frac12\ell(Q)^{\gamma}\ell(R)^{1-\gamma}\geq 2^{-1}2^{r(1-\gamma)}\ell(Q)\geq\lambda\ell(Q),
\end{equation*}
where the last estimate used \eqref{eq:rAtLeast}. Hence \(Q\) is contained in a maximal cube with this property, i.e., in some \(S\in\mathscr{W}\).

Without loss of generality, take \(\Norm{h}{\BMO_{\lambda}^p(\mu)}=1\). By the definition of \(\BMO_{\lambda}^p(\mu)\) and the boundedness of \(b_1\), there holds
\begin{equation*}
  \Norm{1_S\big(h-\ave{h}_S\big)b_1}{L^p(\mu)}^p\leq\mu(\lambda S).
\end{equation*}

Consider the ``Haar'' coefficient \(\pair{1_S\big(h-\ave{h}_S\big)b_1}{\varphi_Q}\). If \(Q\subseteq S\), this coefficient equals \(\pair{h}{b_1\varphi_Q}\), since \(b_1\varphi_Q\) is supported on \(Q\subseteq S\) and has a vanishing integral. Hence it follows from unconditionality that
\begin{equation}\label{eq:BMOuc}
  \BNorm{\sum_{\ontop{Q\in\mathscr{D}_{\good}}{Q\subseteq S}}
    \radem_Q\pair{h}{b_1\varphi_Q}\varphi_Q}{L^p(\prob\otimes\mu)}^p\lesssim
   \mu(\lambda S).
\end{equation}
But the sum over \(S\in\mathscr{W}\) of the left side of the previous estimate coincides with the left side of the assertion by the observation that every \(Q\) there is contained in exactly one \(S\in\mathscr{W}\). On the other hand, the sum over the right hand side is
\begin{equation*}
  \sum_{S\in\mathscr{W}}\mu(\lambda S)=\int\sum_{S\in\mathscr{W}}1_{\lambda S}\ud\mu\lesssim\int 1_R\ud\mu=\mu(R),
\end{equation*}
and this completes the proof.
\end{proof}

Now everything has been prepared for the main result of this section:

\begin{theorem}\label{thm:paraproduct}
Under the standing hypotheses, there holds
\begin{equation*}
  \Norm{\Pi_2 g}{p'}
  \lesssim\Norm{T^* b_2}{\BMO_{\lambda}^{p'}(\mu)}\Norm{g}{p'}
  \lesssim \Norm{g}{p'}.
\end{equation*}
\end{theorem}

By the considerations at the beginning of the section, this shows that
\begin{equation*}
  \Babs{\Exp_{\beta\tilde\beta}
    \sum_{R\in\mathscr{D}_{\good}'}\sum_{\substack{Q\in\mathscr{D}_{\good};Q\subset R \\ \ell(Q)<2^{-r}\ell(R)}}
    \pair{g}{\psi_{R}}\ave{\psi_{R}}_Q\pair{b_2}{T(b_1\varphi_{Q})}\pair{\varphi_{Q}}{f}}
   \lesssim\Norm{g}{p'}\Norm{f}{p}.
\end{equation*}

\begin{proof}
Denote for short $q:=p'$ and
\begin{equation}\label{eq:PhiR}
  \Phi_R:=
     \sum_{\substack{Q\in\mathscr{D}_{\good};Q\subset R \\ \ell(Q)=2^{-r}\ell(R)}}
    \pair{b_2}{T(b_1\varphi_Q)}\frac{1}{\ave{b_2}_R}\cdot\varphi_Q
\end{equation}
Because of the unconditionality of the system \(\{\varphi_Q\}_{Q\in\mathscr{D}}\), it follows that
\begin{equation}\label{eq:PiVsAbstractParapr}
  \Norm{\Pi_2 g}{L^q(\mu;X^*)}
  \lesssim\BNorm{\sum_{R\in\mathscr{D}_{\good}'}\radem_R\Phi_R\ave{g}_{R}}{L^q(\prob\otimes\mu;X^*)},
\end{equation}
and this is of the abstract paraproduct form~\eqref{eq:Pf}, although not yet with the measurability condition required to apply Theorem~\ref{thm:Pf}. This deficit will be repaired with the help of the refined filtrations as in the previous section, and a direct application of the tangent martingale Theorem~\ref{thm:trick} (as opposed to its indirect use via Corollary~\ref{cor:trick} like in the last two sections).

Turning into the details, let
\begin{equation*}
  \mathscr{E}_k:=\{Q\cap R\neq\varnothing;Q\in\mathscr{D}_k,R\in\mathscr{D}_{k+r+1}'\}\supset\{Q\in\mathscr{D}_k;Q^{(1)}\in\mathscr{D}_{\good,k}\}
\end{equation*}
and
\begin{equation*}
  \mathscr{E}_k':=\{Q\cap R\neq\varnothing;Q\in\mathscr{D}_{k+r},R\in\mathscr{D}_k'\}\supset\mathscr{D}_{\good,k}',
\end{equation*}
where the containments are immediate from Lemma~\ref{lem:newGood}. A subsequence of the corresponding $\sigma$-algebras form a filtration, as
\begin{equation*}
  \sigma(\mathscr{E}_{k+r})\subset\sigma(\mathscr{E}_k')\subset\sigma(\mathscr{E}_{k-r-1}).
\end{equation*}

Defining $\Phi_R:=0$ whenever $R\in\mathscr{E}_k'\setminus\mathscr{D}_{\good,k}'$, the estimate in \eqref{eq:PiVsAbstractParapr} may be continued with
\begin{equation*}
\begin{split}
  \Norm{\Pi_2 g}{L^q(\mu;X^*)} 
  &\lesssim\BNorm{\sum_k\radem_k\sum_{R\in\mathscr{E}_{k}'}\Phi_R(x)\ave{g}_{R}}{L^q(\ud\prob(\radem)\ud\mu(x);X^*)} \\
  &\lesssim\sum_{k_0=0}^{2r}\BNorm{\sum_{\substack{k\equiv k_0 \\ \operatorname{mod} 2r+1}}\radem_k
       \sum_{R\in\mathscr{E}_{k}'}\Phi_R(y_R)1_R(x)\ave{g}_{R}}{L^q(\ud\prob(\radem)\ud\mu(x)\ud\nu(y);X^*)},
\end{split}
\end{equation*}
where the measure $\nu=\nu_{k_0}$ on the product space $\prod_{k\equiv k_0}\prod_{R\in\mathscr{E}_k'}R$ (with typical point denoted by $y=(y_R)_R$) is defined as in Theorem~\ref{thm:trick}. The last estimate was an application of the mentioned theorem to each of the subseries with a fixed $k_0$, observing that each $\Phi_R$ is supported on $R\in\mathscr{E}_k'$, zero on $Q\in\mathscr{D}_{\bad,k-r}$ and contant on $Q\in\mathscr{D}_{k-r-1}$ when $Q^{(1)}\in\mathscr{D}_{\good,k-r}$; hence constant on$S\in\mathscr{E}_{k-2r-1}'$.

Now each
\begin{equation*}
  \sum_{R\in\mathscr{E}_{k}'}\Phi_R(y_R)1_R(x),
\end{equation*}
interpreted as a function of $x\in\R^N$ with values in $L^q(\ud\nu(y))$, is obviously $\sigma(\mathscr{E}_k')$-measurable, so that Theorem~\ref{thm:Pf}  is applicable, taking $X_1=X^*$, $X_3=L^q(\nu;X^*)$ and $X_2=L^q(\nu)\subset\bddlin(X_1,X_3)$ in a canonical way. It guarantees that
\begin{equation}\label{eq:applyPf}
\begin{split}
  &\Norm{\Pi_2 g}{L^q(\mu;X^*)}  \\
  &\lesssim\sum_{k_0}\sup_{\substack{k\equiv k_0 \\ S\in\mathscr{E}_k'}}
    \frac{\Norm{g}{L^q(\mu;X^*)}}{\mu(S)^{1/q}}
      \BNorm{\sum_{\substack{j\equiv k_0 \\ j\leq k}}\radem_j\sum_{\substack{R\in\mathscr{E}_j' \\ R\subseteq S}}\Phi_R(y_R)1_R(x)}{
        L^q(\ud\prob(\radem)\ud\mu(x)\ud\nu(y))}
\end{split}
\end{equation}
and another application of Theorem~\ref{thm:trick} permits the replacement of $\Phi_R(y_R)1_R(x)$ by simply $\Phi_R(x)$. Also, the summation condition $R\in\mathscr{E}_j'$ may obviously be replaced by $R\in\mathscr{D}_{\good,j}'$, as the remaining terms are zero by definition. And then one rearranges the summation in terms of the maximal cubes $P\in\mathscr{D}_{\good,j}'$ appearing in this summation. The additivity of the integral on disjointly supported functions implies that
\begin{equation*}
\begin{split}
  &\BNorm{\sum_{\substack{j\equiv k_0 \\ j\leq k}}\radem_j\sum_{\substack{R\in\mathscr{E}_j' \\ R\subseteq S}}\Phi_R}{
        L^q(\prob\otimes\mu)}^q
  =\sum_P \BNorm{\sum_{\substack{R\in\mathscr{D}_{\good}' \\ \log_2\ell(R)\equiv k_0 \\ R\subseteq P}}\radem_R\Phi_R}{
        L^q(\prob\otimes\mu)}^q \\
   &\leq\sum_P\BNorm{\sum_{\substack{Q\in\mathscr{D}_{\good};\,Q\subset P \\ \ell(Q)\leq 2^{-r}\ell(P) \\ \log_2\ell(Q)\equiv k_0-r}}
        \radem_Q\pair{T^*b_2}{b_1\varphi_Q}\varphi_Q}{
        L^q(\prob\otimes\mu)}^q,
\end{split}
\end{equation*}
where the last line was essentially just writing out the definition of each $\Phi_R$. Now Lemma~\ref{lem:NTV51} and the maximality of the cubes $P$, all of which are contained in $S$, make the estimate continue with
\begin{equation*}
  \lesssim\sum_P\mu(P)\Norm{T^*b_2}{\BMO_{\lambda}^q(\mu)}^q\leq\mu(S)\Norm{T^*b_2}{\BMO_{\lambda}^q(\mu)}^q.
\end{equation*}

The proof is completed by recalling from Nazarov et al.~\cite{NTV:Tb}, Section~2.3, that the assumption \(\Norm{T^* b_2}{\BMO_{\lambda}^1(\mu)}\leq 1\), combined with the other hypotheses of \(Tb\) theorem~\ref{thm:main}, already implies that \(\Norm{T^* b_2}{\BMO_{\lambda}^q(\mu)}\lesssim 1\) for all \(q\in[1,\infty)\) (which would not be true for an arbitrary \(h\in\BMO_{\lambda}^1(\mu)\) in place of \(T^* b_2\)), and substituting back to \eqref{eq:applyPf}.
\end{proof}


\section{Close-by cubes of comparable size}\label{sec:close}

The part of the series \eqref{eq:matrixT} which has not been addressed so far consists of the pairs of good cubes \(Q,R\) which are close to each other both in terms of their position and size; more precisely, \(2^{-r}\ell(R)\leq\ell(Q)\leq 2^r\ell(R)\) and \(\dist(Q,R)<\ell(Q)\wedge\ell(R)\). In this section, a certain portion, determined by a new auxiliary parameter $\eta$, of this remaining part will be estimated, and accordingly, the implicit constants here are allowed to depend on both $r$ and $\eta$, in addition to the parameters listed in Notation~\ref{not:C}. Only the size, and not the cancellation, properties of the ``Haar'' functions will be exploited here, so the estimates are equally valid for both types of functions, $\varphi_{Q,0}^{b}$ and $\varphi_{Q,u}^b$, $u\geq 1$, appearing in \eqref{eq:matrixT}.

Given \(R\in\mathscr{D}_{\good}'\), there are only boundedly many cubes \(Q\in\mathscr{D}_{\good}\) like this, and thus it remains to consider a finite number of subseries
\begin{equation}\label{eq:closeSeries}
  \sum_{R\in\mathscr{D}_{\good}'}
    \pair{g}{\psi_R}T_{RQ}\pair{\varphi_Q}{f}
  =\sum_{R\in\mathscr{D}_{\good}'}
    \pair{g}{\psi_R}\pair{\psi_R b_2}{T(b_1\varphi_Q)}\pair{\varphi_Q}{f},
\end{equation}
where \(Q=Q(R)\). Fix one such series; the convention that \(Q\) is implicitly a function of \(R\) will be maintained without further notice throughout the rest of this section. Without essential loss of generality, it is permissible to act as if the map \(R\mapsto Q(R)\) was invertible, so that the same series \eqref{eq:closeSeries} could also be written with the summation variable \(Q\in\mathscr{D}_{\good}\), with \(R=R(Q)\). In reality, it may happen that some \(Q\) has no preimage \(R\), or that there are several preimages. But in the first case one may simply interpret the corresponding terms as zero, and in the second case the number of preimages is nevertheless bounded, so that one can always split the summations under consideration into boundedly many subseries and proceed with the triangle inequality; such technical details will not be indicated explitly.

Observing that
\begin{equation*}
  b_1\varphi_Q\pair{\varphi_Q}{f}
  =\sum_{\ontop{Q'\in\mathscr{D},Q'\subset Q}{\ell(Q')=\ell(Q)/2}}
     b_1 1_{Q'}\ave{\varphi_Q}_{Q'}\pair{\varphi_Q}{f}
  =:\sum_{\ontop{Q'\in\mathscr{D},Q'\subset Q}{\ell(Q')=\ell(Q)/2}}
     b_1 1_{Q'}c_{Q'}(f)
\end{equation*}
and similarly
\begin{equation*}
  b_2\psi_R\pair{\psi_R}{g}
  =\sum_{\ontop{R'\in\mathscr{D}',R'\subset R}{\ell(R')=\ell(R)/2}}
     b_2 1_{R'}d_{R'}(g),
\end{equation*}
the series \eqref{eq:closeSeries} splits into \((2^N)^2\) subseries of the form
\begin{equation}\label{eq:closeReduct}
  \sum_{R\in\mathscr{D}'}d_R(g)\pair{1_R b_2}{T(b_1 1_Q)}c_Q(f),
\end{equation}
where \(Q=Q(R)\) is a possibly different function of \(R\) from the one before, but still with the property that \(2^{-r}\ell(R)\leq\ell(Q)\leq 2^r\ell(R)\).

As in \cite{NTV:Tb}, for each cube \(Q\), define the boundary region
\begin{equation}\label{eq:bdryReg}
  \delta_Q:=(1+2\eta)Q\setminus(1-2\eta)Q,
\end{equation}
where the new auxiliary parameter \(\eta>0\) is to be chosen. Then, for each \(Q\in\mathscr{D}\), its bad part is defined by
\begin{equation*}
  Q_{\bad}:=Q\cap\Big(\bigcup_{\ontop{R\in\mathscr{D}'}{2^{-r}\leq\ell(R)/\ell(Q)\leq 2^r}}\delta_R\Big),
\end{equation*}
while for \(R\in\mathscr{D}'\) a similar definition with the obvious modification is made. (Note that, just like in~\cite{NTV:Tb}, this is different badness from the one considered in the previous sections: some cubes, as entities, are good while some are bad, but all cubes, whether good or bad, have their bad part in the sense of the above definition.)

Given \(R\in\mathscr{D}'\) and \(Q=Q(R)\in\mathscr{D}\) appearing in the sum \eqref{eq:closeReduct}, let
\begin{equation*}
  \Delta:=Q\cap R,\quad\begin{matrix}
  Q_{\sep}:=Q\setminus\Delta\setminus\delta_R, &
  Q_{\partial}:=(Q\setminus\Delta)\cap\delta_R\subseteq Q_{\bad}, \\
  R_{\sep}:=R\setminus\Delta\setminus\delta_Q, &
  R_{\partial}:=(R\setminus\Delta)\cap\delta_Q\subseteq R_{\bad}
  \end{matrix}  
\end{equation*}
so that there are disjoint unions
\begin{equation*}
  Q=\Delta\cup Q_{\sep}\cup Q_{\partial},\quad
  R=\Delta\cup R_{\sep}\cup R_{\partial}.
\end{equation*}
Then the matrix coefficient in \eqref{eq:closeReduct} can be written as
\begin{equation}\label{eq:fiveTerms}\begin{split}
  \pair{1_R b_2}{T(b_1 1_Q)}
  = &\pair{1_{R_{\sep}}b_2}{T(b_1 1_Q)}
   +\pair{1_{R_{\partial}}b_2}{T(b_1 1_Q)} \\
   &\qquad+\pair{1_{\Delta}b_2}{T(b_1 1_{\Delta})} \\
   &+\pair{1_{\Delta}b_2}{T(b_1 1_{Q_{\partial}})}
   +\pair{1_{\Delta}b_2}{T(b_1 1_{Q_{\sep}})}.
\end{split}\end{equation}

The \emph{second} and the \emph{fourth terms} on the right of \eqref{eq:fiveTerms} correspond to the bad parts, and will be left alone for a while. The \emph{middle term} satisfies
\begin{equation*}
  \pair{1_{\Delta}b_2}{T(b_1 1_{\Delta})}
  =:T_{\Delta}\mu(\Delta),\qquad
  \abs{T_{\Delta}}\leq 1,
\end{equation*}
as a direct application of the assumed rectangular weak boundedness property of $M_{b_2}TM_{b_1}$, since \(\Delta=Q\cap R\) is clearly a rectangle. Hence
\begin{equation*}\begin{split}
  &\Babs{\sum_R d_R(g)\pair{1_{\Delta}b_2}{T(b_1 1_{\Delta})} c_Q(f)}
  =\Babs{\int\sum_R 1_R d_R(g)T_{\Delta}c_Q(f) 1_Q\ud\mu} \\
  &=\Babs{\iint\sum_R\radem_R 1_R d_R(g)\sum_{R'}\radem_{R'} T_{\Delta}c_{Q(R')}(f)1_{Q(R')}\ud\prob(\radem)\ud\mu} \\
  &\leq\BNorm{\sum_R\radem_R 1_R d_R(g)}{L^{p'}(\prob\otimes\mu;X^*)}
   \BNorm{\sum_Q\radem_Q T_{\Delta}c_Q(f)1_Q}{L^p(\prob\otimes\mu;X)} \\
  &\lesssim\BNorm{\sum_R\radem_R \psi_{R^{(1)}}\pair{\psi_{R^{(1)}}}{g}}{L^{p'}(\prob\otimes\mu;X^*)}
    \BNorm{\sum_Q\radem_Q\varphi_{Q^{(1)}}\pair{\varphi_{Q^{(1)}}}{f}}{L^p(\prob\otimes\mu;X)} \\
  &\lesssim\Norm{g}{L^{p'}(\mu;X^*)}\Norm{f}{L^p(\mu;X)},
\end{split}\end{equation*}
where, in the second to last step, the contraction principle was used both to remove the bounded factors \(T_{\Delta}\) and to dominate the functions \(1_Q c_Q(f)=1_Q\varphi_{Q^{(1)}}\pair{\varphi_{Q^{(1)}}}{f}\) by the right-hand side without the \(1_Q\), and similarly on the \(g\) side.

Now consider the \emph{first term} on the right of \eqref{eq:fiveTerms}; the fifth terms is essentially similar, the main point being that the two indicators in both terms correspond to sets separated from each other. By \eqref{eq:CZK1},
\begin{equation*}\begin{split}
  \abs{\pair{1_{R_{\sep}}b_2}{T(b_1 1_Q)}}
  &=\Babs{\int_{R_{\sep}}\int_Q b_2(x)K(x,y)b_1(y)\ud\mu(y)\ud\mu(x)} \\
  &\lesssim\frac{\mu(R_{\sep})\mu(Q)}{\dist(R_{\sep},Q)^d}
  \lesssim\frac{\mu(R)\mu(Q)}{\ell(Q)^d}.
\end{split}\end{equation*}

Write
\begin{equation*}
  \pair{1_{R_{\sep}}b_2}{T(b_1 1_Q)}=:T_Q\frac{\mu(R)\mu(Q)}{\ell(Q)^d},\qquad
  \abs{T_Q}\lesssim 1.
\end{equation*}
Then
\begin{equation*}\begin{split}
  &\sum_R d_R(g)\pair{1_{R_{\sep}}b_2}{T(b_1 1_Q)}c_Q(f)  \\
  &=\sum_R\pair{g}{\psi_{R^{(1)}}}\ave{\psi_{R^{(1)}}}_R\,\mu(R)\frac{T_Q}{\ell(Q)^d}
     \mu(Q)\ave{\varphi_{Q^{(1)}}}_Q\pair{\varphi_{Q^{(1)}}}{f} \\
  &=:\sum_R\pair{g}{\psi_{R^{(1)}}}\Norm{\psi_{R^{(1)}}}{L^1(\mu)}\frac{\tilde{T}_Q}{\ell(Q)^d}
     \Norm{\varphi_{Q^{(1)}}}{L^1(\mu)}\pair{\varphi_{Q^{(1)}}}{f},
\end{split}\end{equation*}
where also \(\abs{\tilde{T}_Q}\lesssim 1\). Reindexing the sum, so as to write simply \(Q\) and \(R\) instead of \(Q^{(1)}\) and \(R^{(1)}\), reduces the considerations to the series
\begin{equation}\label{eq:closeManip}
  \sum_R\pair{g}{\psi_R}\Norm{\psi_R}{L^1(\mu)}\frac{t_Q}{\ell(Q)^d}\Norm{\varphi_Q}{L^1(\mu)}\pair{\varphi_Q}{f},
\end{equation}
where \(\abs{t_Q}\lesssim 1\). By \eqref{eq:firstBy},
\begin{equation}\label{eq:closeMore}\begin{split}
  \frac{\abs{\eqref{eq:closeManip}}}{\Norm{g}{L^{p'}(\mu;X^*)}}
  &\leq
    \BNorm{\sum_Q\radem_Q \psi_{R(Q)}\Norm{\psi_{R(Q)}}{1}\frac{t_Q}{\ell(Q)^d}
       \Norm{\varphi_Q}{1}\pair{\varphi_Q}{f}}{L^p(\prob\otimes\mu;X)} \\
  &\lesssim\BNorm{\sum_Q\radem_Q 1_{R(Q)}\frac{1}{\ell(Q)^d}\int_Q
       \Norm{\varphi_Q}{1}\varphi_Q(y)f(y)\ud\mu(y)}{L^p(\prob\otimes\mu;X)}.
\end{split}\end{equation}

By symmetry, one may assume that \(\ell(Q)\geq\ell(R)\), hence \(\ell(Q^{(\theta(0))})>2^r\ell(R)\).
In order to apply the tangent martingale trick, one checks that \(R\subset Q^{(r)}\). Indeed, if not, then
\begin{equation*}\begin{split}
  \ell(R)\geq\dist(R,Q)
  &\geq\dist(R,Q^{(r)})
  =\dist(R,\partial Q^{(r)}) \\
  &\geq\ell(R)^{\gamma}(\ell(Q^{(r)}))^{1-\gamma}
  \geq 2^{r(1-\gamma)}\ell(R)>\ell(R),
\end{split}\end{equation*}
a contradiction.

Hence, reindexing the summation in terms of \(S=Q^{(r)}\),
\begin{equation*}
  RHS\eqref{eq:closeMore}
  \lesssim\BNorm{\sum_{k\in\Z}\radem_k\sum_{S\in\mathscr{D}_k}\frac{1_S(x)}{\mu(S)}
     \int_S K_S(x,y)\frac{1_S\D_{k-r}^{b_1}f}{b_1}(y)\ud\mu(y)}{L^p(\ud\prob(\radem)\ud\mu(x);X)},
\end{equation*}
where
\begin{equation*}
  K_S(x,y)=\sum_{\ontop{Q\in\mathscr{D}^{\good}_{k-r}}{Q\subset S}}
   1_{R(Q)}(x)\Norm{\varphi_Q}{L^1(\mu)}\varphi_Q(y)b_1(y)
\end{equation*}
is supported in \(S\times S\) and \(\Norm{K_S}{\infty}\lesssim 1\). As before, \(b_1^{-1}\D_{k-r}^{b_1}f\) is constant on all \(Q'\in\mathscr{D}_{k-r-1}\), so that splitting the \(k\) summation into \(r+1\) subseries according to \(k\equiv k_0\mod r+1\), and Corollary~\ref{cor:trick} applies to each of these. The conclusion is
\begin{equation*}
  \Babs{\sum_R d_R(g)\pair{1_{R_{\sep}}b_2}{T(b_1 1_Q)}c_Q(f)}
  \lesssim\Norm{g}{L^{p'}(\mu;X^*)}\Norm{f}{L^p(\mu;X)},
\end{equation*}
and the same is true with \(1_{R_{\sep}}\) and \(1_Q\) replaced by \(1_{\Delta}\) and \(1_{Q_{\sep}}\), as argued above.


\section{Bad boundary regions}\label{sec:bad}

It is time to encounter the bad parts, which were avoided until now, in order to complete the proof of \(Tb\) theorem~\ref{thm:main}. In this section, the implicit constant are still allowed to depend on $r$, as before, but any dependence on the auxiliary parameter $\eta$ from the previous section, which was used to define the depth of the boundary regions \(\delta_Q\) in \eqref{eq:bdryReg}, will be stated explicitly. In estimating the expansion
\begin{equation}\label{eq:seriesRecalled}
  \sum_{Q\in\mathscr{D},R\in\mathscr{D}'}\pair{g}{\psi_R}T_{RQ}\pair{\varphi_Q}{f},
\end{equation}
the following inequalities, following the convention about the implicit constants just stated, have been obtained so far:
\begin{equation}\label{eq:handled}\begin{split}
 \Babs{\sum_{R\in\mathscr{D}_{\good}'}
  \sum_{\ontop{Q\in\mathscr{D}_{\good}}{\ell(Q)\leq\dist(Q,R)\wedge\ell(R)}}
   \pair{g}{\psi_R}T_{RQ}\pair{\varphi_Q}{f}}
 &\leq C\Norm{g}{p'}\Norm{f}{p}, \\
 \Babs{\sum_{R\in\mathscr{D}_{\good}'}
  \sum_{\ontop{Q\in\mathscr{D}_{\good},\,Q\subset R}{\ell(Q)<2^{-r}\ell(R)}}
   \pair{g}{\psi_R}\tilde{T}_{RQ}\pair{\varphi_Q}{f}}
 &\leq C\Norm{g}{p'}\Norm{f}{p}, \\
 \Babs{\sum_{R\in\mathscr{D}_{\good}'}
  \sum_{\substack{ Q\in\mathscr{D}_{\good}\\ 
         \dist(Q,R)<\ell(Q)\wedge\ell(R) \\ 2^{-r}\leq\ell(Q)/\ell(R)\leq 2^r}}
   \pair{g}{\psi_R}T_{RQ}^{\good}\pair{\varphi_Q}{f}}
 &\leq C_{\eta}\Norm{g}{p'}\Norm{f}{p}, \\
\end{split}\end{equation}
where \(T_{RQ}^{\good}\) is the part of the coefficient \(T_{RQ}\) corresponding to the first, third and fifth terms in \eqref{eq:fiveTerms} in the decomposition performed in the previous section, and $\tilde{T}_{QR}$ is the modified matrix entry with the paraproduct removed, as treated in Section~\ref{sec:contained}. In addition, it was shown in Section~\ref{sec:para} that the correction terms $T_{QR}-\tilde{T}_{QR}$ satisfy the similar bound on average, in the sense that
\begin{equation}\label{eq:handledOnAve}
  \Babs{\Exp_{\beta\tilde\beta}\sum_{R\in\mathscr{D}_{\good}'}
   \sum_{\ontop{Q\in\mathscr{D}_{\good},\,Q\subset R}{\ell(Q)<2^{-r}\ell(R)}}
   \pair{g}{\psi_R}(T_{QR}-\tilde{T}_{RQ})\pair{\varphi_Q}{f}}
 \leq C\Norm{g}{p'}\Norm{f}{p}.
\end{equation}

Note that, were it not for the labels ``good'' in various places in \eqref{eq:handled} and \eqref{eq:handledOnAve}, these subseries would cover the half of \eqref{eq:seriesRecalled} with \(\ell(Q)\leq\ell(R)\), and in fact a bit more in the case of close-by cubes. By symmetry, it hence remains to treat the bad cubes, and also the bad parts of the matrix coefficients, \(T_{RQ}^{\bad}=T_{RQ}-T_{RQ}^{\good}\) corresponding to the second and fourth terms in \eqref{eq:fiveTerms}, which were left out in the last line of \eqref{eq:handled}.

The treatment of these remaining terms will be similar to the estimate \eqref{eq:handledOnAve}, in that control is gained only after averaging over the dyadic systems. But there is also the important difference, that the bounds will now depend on the operator norm $\Norm{T}{\bddlin(L^p(\mu;X))}$, and one needs to get a small factor in front of it in order to eventually absorb it into the left side of the final inequality.

The estimation of the bad parts will be based on the fact that every UMD space has cotype \(s\) for some \(s\in[2,\infty)\), i.e., satisfies the inequality:
\begin{equation*}
  \Big(\sum_{j=1}^n\norm{\xi_j}{X}^s\Big)^{1/s}
  \lesssim\BNorm{\sum_{j=1}^n\radem_j\xi_j}{L^2(\Omega;X)},
\end{equation*}
and then on the following improvement of the contraction principle (corresponding to \(t=\infty\) below) under this extra condition. Note that the previous estimate (with the usual modification) is always true for $s=\infty$ and never for $s<2$.

\begin{proposition}\label{prop:DJT}
Let \(\xi_j\in X\), where $X$ is a Banach space of cotype \(s\in[2,\infty)\) and let \(\theta_j\in L^t(\tilde\Omega)\) for some \(\sigma\)-finite measure space \(\tilde\Omega\) and \(t\in(s,\infty)\). Then
\begin{equation*}
  \BNorm{\sum_{j=1}^{\infty}\radem_j\theta_j\xi_j}{L^t(\tilde\Omega;L^2(\Omega;X))}
  \lesssim\sup_j\Norm{\theta_j}{t}\BNorm{\sum_{j=1}^{\infty}\radem_j\xi_j}{L^2(\Omega;X)}.
\end{equation*}
\end{proposition}

\begin{proof}
By approximation, it suffices to consider finite sums \(1\leq j\leq n\). This result can be found in \cite{HV}, Lemma~3.1.
\end{proof}

Now turn to the bad analogue of the last series in \eqref{eq:handled}, and more precisely to the part of the series \eqref{eq:closeReduct} with the second term from \eqref{eq:fiveTerms}, \(\pair{1_{R_{\partial}}b_2}{T(b_1 1_{Q})}\), in place of \(\pair{1_R b_2}{T(b_1 1_Q)}\).

\begin{lemma}\label{lem:badBdry1}
Let \(X^*\) have cotype \(s\) and take \(t>s\vee p'\). Then
\begin{equation*}\begin{split}
  \Exp_{\beta} &\Babs{\sum_{R\in\mathscr{D}_{\good}'}d_R(g)\pair{1_{R_{\partial}}b_2}{T(b_1 1_{Q})}c_Q(f)} \\
  &\lesssim \eta^{1/t}\Norm{T}{\bddlin(L^p(\mu;X))}\Norm{g}{L^{p'}(\mu;X^*)}\Norm{f}{L^p(\mu;X)}.
\end{split}\end{equation*}
\end{lemma}

\begin{proof}
First randomize and use H\"older to the result that
\begin{equation*}\begin{split}
  &\Babs{\sum_{R\in\mathscr{D}_{\good}'}d_R(g)\pair{1_{R_{\partial}}b_2}{T(b_1 1_{Q})}c_Q(f)} \\
  &=\Babs{\int_{\Omega}\Bpair{\sum_{S\in\mathscr{D}_{\good}'}\radem_S d_S(g)1_{S_{\partial}}b_2}{T\Big(
     \sum_{R\in\mathscr{D}_{\good}'}\radem_R c_Q(f) b_1 1_{Q}\Big)}\ud\prob(\radem)} \\
  &\leq\BNorm{\sum_{S\in\mathscr{D}'}\radem_S d_S(g)1_{S_{\partial}}b_2}{L^{p'}(\prob\otimes\mu;X^*)}
    \BNorm{T\Big(\sum_{R\in\mathscr{D}_{\good}'}\radem_R c_Q(f) b_1 1_{Q}\Big)}{L^p(\prob\otimes\mu;X)}.
\end{split}\end{equation*}
From the second factor, one may extract \(\Norm{T}{\bddlin(L^p(\mu;X))}\), and then by the contraction principle and unconditionality
\begin{equation*}
\begin{split}
  \BNorm{\sum_{Q\in\mathscr{D}}\radem_Q c_Q(f)b_1 1_Q}{L^p(\prob\otimes\mu;X)} 
  &\leq\BNorm{\sum_{Q\in\mathscr{D}}\radem_Q b_1\varphi_{Q^{(1)}}\pair{\varphi_{Q^{(1)}}}{f}}{L^p(\prob\otimes\mu;X)} \\
  &\lesssim\Norm{f}{L^p(\mu;X)}.
\end{split}
\end{equation*}

As for the first factor, write
\begin{equation*}
  \delta(k):=\bigcup_{j=k-r}^{k+r}\bigcup_{Q\in\mathscr{D}_j}\delta_Q,
\end{equation*}
and then (dropping $b_2$ by the contraction principle)
\begin{equation*}\begin{split}
  &\Exp_{\beta}
    \BNorm{\sum_{S\in\mathscr{D}'}\radem_S d_S(g)1_{S_{\partial}}}{L^{p'}(\prob\otimes\mu;X^*)}
  =\Exp_{\beta}
    \BNorm{\sum_{k\in\Z}\radem_k 1_{\delta(k)}\sum_{R\in\mathscr{D}_k'}d_R(g)1_R}{L^{p'}(\prob\otimes\mu;X^*)} \\
  &\leq\Big(\int_{\R^N}\Big[\Exp_{\beta}\BNorm{
      \sum_{k\in\Z}\radem_k 1_{\delta(k)}(x)\sum_{R\in\mathscr{D}_k'}d_R(g)1_R(x)}{L^{p'}(\prob;X^*)}^t
    \Big]^{p'/t}\ud\mu(x)\Big)^{1/p'}
\end{split}\end{equation*}
for \(t\geq p'\).

For each fixed \(x\in\R^N\), the integrand is of the form considered in Proposition~\ref{prop:DJT}, with 
\begin{equation*}
  \xi_k=\sum_{R\in\mathscr{D}_k'}d_R(g)1_R(x)=d_{R(x,k)}(g),
\end{equation*}
where \(R(x,k)\) is the unique \(R\in\mathscr{D}_k'\) containing \(x\). (There is now an \(L^{p'}\) norm instead of the \(L^2\) norm on the probability space \((\Omega,\prob)\), which is however irrelevant thanks to Kahane's inequality.) The random variables \(1_{\delta(k)}(x)\)---as functions of the implicit variable $\beta\in(\{0,1\}^N)^{\Z}$, which governs the distribution of the random dyadic system \(\mathscr{D}\), and hence of the boundary regions $\delta(k)$---obviously belong to all \(L^t((\{0,1\}^N)^{\Z})\) for all \(t\in[1,\infty]\), and satisfy
\begin{equation*}
  \Norm{1_{\delta(k)}(x)}{L^t((\{0,1\}^N)^{\Z})}
  =\prob_{\beta}(1_{\delta(k)}(x)=1)^{1/t}\lesssim\eta^{1/t}.
\end{equation*}
With a choice of \(t\) as in the assertion, Proposition~\ref{prop:DJT} then implies that
\begin{equation*}
  \Exp_{\beta}\BNorm{\sum_{S\in\mathscr{D}'}\radem_S d_S(g)1_{S_{\partial}}}{L^{p'}(\prob\otimes\mu;X^*)}
  \lesssim\eta^{1/t}\BNorm{\sum_{R\in\mathscr{D}'}\radem_R d_R(g)1_R}{L^{p'}(\prob\otimes\mu;X^*)},
\end{equation*}
and this is dominated by \(\eta^{1/t}\Norm{g}{L^{p'}(\mu;X^*)}\) by similar contraction principle and unconditionality arguments as before.
\end{proof}

The case of the fourth term from \eqref{eq:fiveTerms} is analogous (the only break in the symmetry being one more application of the contraction principle to estimate \(1_{\Delta}\) by \(1_R\) in the appropriate place), and I only state the result, leaving its verification as an easy excercise along the lines of the previous proof.

\begin{lemma}\label{lem:badBdry2}
Let \(X\) have cotype \(s\) and take \(t>s\vee p\). Then
\begin{equation*}\begin{split}
  \Exp_{\beta'} &\Babs{\sum_{Q\in\mathscr{D}_{\good}}d_R(g)\pair{1_{\Delta}b_2}{T(b_1 1_{Q_{\partial}})}c_Q(f)} \\
  &\lesssim \eta^{1/t}\Norm{T}{\bddlin(L^p(\mu;X))}\Norm{g}{L^{p'}(\mu;X^*)}\Norm{f}{L^p(\mu;X)}.
\end{split}\end{equation*}
\end{lemma}

The results of Lemmas~\ref{lem:badBdry1} and \ref{lem:badBdry2} may be summarized as
\begin{equation}\label{eq:handledBdry}
\begin{split}
 &\Babs{\Exp_{\beta\beta'}\sum_{R\in\mathscr{D}_{\good}'}
  \sum_{\substack{ Q\in\mathscr{D}_{\good}\\ 
         \dist(Q,R)<\ell(Q)\wedge\ell(R) \\ 2^{-r}\leq\ell(Q)/\ell(R)\leq 2^r}}
   \pair{g}{\psi_R}T_{RQ}^{\bad}\pair{\varphi_Q}{f}} \\
 &\qquad\qquad\qquad\qquad\qquad\qquad\leq C\eta^{1/t}\Norm{T}{\bddlin(L^p(\mu;X))}\Norm{g}{p'}\Norm{f}{p},
\end{split}
\end{equation}
for any $t$, which is bigger than both $\max(p,p')$ and the cotypes of $X$ and $X^*$.


\section{Synthesis}

The proof of $Tb$ theorem~\ref{thm:main} will now be completed. This also involves choosing appropriate values for the auxiliary parameters $r$ and $\eta$. Hence any dependence on these numbers will now be indicated explicitly, and any constant $C$ may only depend on the parameters as listed in Notation~\ref{not:C}.

Given a function \(f\in L^p(\mu;X)\), define its good and bad parts
\begin{equation*}
  f_{\lambda}:=\sum_{\substack{Q\in\mathscr{D}_{\lambda} \\ \ell(Q)\leq 2^m}}\D_Q^{b_1}f+
     \sum_{\substack{Q\in\mathscr{D}_{\lambda} \\ \ell(Q)=2^m}}\Exp_Q^{b_1}f,\qquad
  \lambda\in\{\good,\bad\};
\end{equation*}
an analogous definition is made for \(g\in L^{p'}(\mu;X^*)\). Observe that the decomposition $f=f_{\good}+f_{\bad}$ depends on the random parameters $\beta$ (which determines the dyadic system $\mathscr{D}$) as well as $\beta'$ and $\tilde\beta'$ (which determine the goodness or badness of a given $Q\in\mathscr{D}$); likewise, the splitting $g=g_{\good}+g_{\bad}$ depends on $\beta'$, $\beta$ and $\tilde\beta$.

Having assumed that $T$ is a bounded operator, fix some compactly supported $f\in L^p(\mu;X)$ and $g\in L^{p'}(\mu;X^*)$ so that
\begin{equation*}
  \Norm{T}{\bddlin(L^p(\mu;X))}\leq 2\pair{g}{Tf},\qquad\Norm{f}{p}=\Norm{g}{p'}=1.
\end{equation*}
One first expands
\begin{equation*}
  \pair{g}{Tf}=\pair{g_{\good}}{Tf_{\good}}+\pair{g_{\good}}{Tf_{\bad}}+\pair{g_{\bad}}{Tf}.
\end{equation*}
The left side is independent of the parameters $\beta,\beta',\tilde\beta,\tilde\beta'$, which govern the splitting on the right. Now take the expectation $\Exp_{\beta\beta'\tilde\beta\tilde\beta'}$ of both sides of the previous equality.

Collecting the estimates from the previous sections, summarized in \eqref{eq:handled}, \eqref{eq:handledOnAve} and \eqref{eq:handledBdry}, and writing out the dependence on $r$, which was suppressed in the implicit constants until now, it has been established that
\begin{equation*}
  \abs{\Exp_{\beta\beta'\tilde\beta\tilde\beta'}\pair{g_{\good}}{Tf_{\good}}}
  \leq C(\eta,r)+C(r)\eta^{1/t}\Norm{T}{\bddlin(L^p(\mu;X))}.
\end{equation*}
Indeed, the mentioned estimates involved averaging over some of the $\beta$ parameters only, being uniform with respect to the other ones, but such inequalities clearly imply the weaker version, where all these parameters are averaged out. Altogether, using also that $\Norm{g_{\good}}{p'}\lesssim\Norm{g}{p'}=1$, it follows that
\begin{equation}\label{eq:almostThere}
\begin{split}
  \Norm{T}{\bddlin(L^p(\mu;X))}&\leq C(\eta,r) +C(r)\eta^{1/t}\Norm{T}{\bddlin(L^p(\mu;X))} \\
    &\quad+C\Norm{T}{\bddlin(L^p(\mu;X))}\big(\Exp_{\beta'}\Exp_{\beta\tilde\beta}\Norm{g_{\bad}}{p'}
      +\Exp_{\beta}\Exp_{\beta'\tilde\beta'}\Norm{f_{\bad}}{p}\big).
\end{split}
\end{equation}

It remains to estimate the expectations of the bad parts, which is very similar to the previous section, for instance the second one:
\begin{equation*}\begin{split}
  &\Exp_{\beta'\tilde\beta'}\Norm{f_{\bad}}{L^p(\mu;X)} \\
  &\lesssim\Big(\int_{\R^N}\Big[\Exp_{\beta'\tilde\beta'}\BNorm{
      \sum_{Q\in\mathscr{D}}\radem_Q 1_{\bad}^{\beta'\tilde\beta'}(Q)\,\D_Q^{b_1}f(x)}{
         L^p(\prob;X)}^t\Big]^{p/t}\ud\mu(x)\Big)^{1/p}
\end{split}\end{equation*}
By Lemma~\ref{lem:NTV92}, the random variables $(\beta',\tilde\beta')\mapsto 1_{\bad}^{\beta'\tilde\beta'}(Q)$ satisfy
\begin{equation*}
  \Norm{1_{\bad}(Q)}{L^t((\{0,1\}^{2N})^{\Z})}
  =\prob_{\beta'\tilde\beta'}(Q\text{ is bad})^{1/t}=\epsilon(r),
\end{equation*}
where \(\epsilon(r)\to 0\) as \(r\to\infty\). Hence, by Proposition~\ref{prop:DJT}, with \(t>s\vee p\), where \(X\) has cotype \(s\), it follows that
\begin{equation*}
  \Exp_{\beta'\tilde\beta'}\Norm{f_{\bad}}{L^p(\mu;X)}
  \lesssim \epsilon(r)\BNorm{\sum_{Q\in\mathscr{D}}\radem_Q\D_Q^{b_1}f}{L^p(\prob\otimes\mu;X)}
  \lesssim \epsilon(r)\Norm{f}{L^p(\mu;X)}.
\end{equation*}

Using the similar estimate for \(g_{\bad}\), and substituting back to \eqref{eq:almostThere}, it follows that
\begin{equation*}
  \Norm{T}{\bddlin(L^p(\mu;X))}
  \leq C(r,\eta)+C(r)\eta^{1/t}\Norm{T}{\bddlin(L^p(\mu;X))}
      +C\epsilon(r)\Norm{T}{\bddlin(L^p(\mu;X))}.
\end{equation*}
Now one first fixes a large enough \(r\) so that \(C\epsilon(r)<\frac{1}{3}\). Then one picks a small enough \(\eta\) so that \(C(r)\eta^{1/t}<\frac{1}{3}\). Thus
\begin{equation*}
  \Norm{T}{\bddlin(L^p(\mu;X))}
  \leq C(r,\eta)+\big(\frac{1}{3}+\frac{1}{3}\big)\Norm{T}{\bddlin(L^p(\mu;X))},
\end{equation*}
and this completes the proof of \(Tb\) theorem~\ref{thm:main}.


\section{Operator-valued kernels}\label{sec:operator}

This section explains the extension of \(Tb\) theorem~\ref{thm:main} to the case of operator-valued kernels \(K(x,y)\in\bddlin(X)\), as stated in \(Tb\) theorem~\ref{thm:operator}.  Following the ``Rademacher rule of thumb'' for operator-kernels mentioned in the Introduction, define a \emph{\(d\)-dimensional Rademacher--Calder\'on--Zygmund kernel} as a function \(K(x,y)\) of variables \(x,y\in\R^N\) with \(x\neq y\) and taking values in \(\bddlin(X)\), which satisfies
\begin{equation}\label{eq:RCZK1}
  \rbound\big(\{\abs{x-y}^d K(x,y):x,y\in\R^N,x\neq y\}\big)\leq 1,
\end{equation}
\begin{equation}\label{eq:RCZK2}
\begin{split}
    \rbound\Big(\Big\{\frac{\abs{x-y}^{d+\alpha}}{\abs{x-x'}^{\alpha}} [K(x,y)-K(x',y)],
        \frac{\abs{x-y}^{d+\alpha}}{\abs{x-x'}^{\alpha}}[K(y,x)-K(y,x')]:& \\
    x,x',y\in\R^N,\abs{x-y}>2\abs{x-x'}>0\Big\}&\Big)\leq 1
\end{split}
\end{equation}
for some \(\alpha>0\). Recall that \(\rbound(\mathscr{T})\) designates the Rademacher-bound of the set \(\mathscr{T}\), as defined after~\eqref{eq:defRbd}. As in the scalar case, multiplicative constants could be allowed in these conditions, but will be supressed.

Let \(T:f\mapsto Tf\) be a linear operator acting on some functions \(f:\R^N\to X\) or \(f:\R^N\to\C\), producing new functions \(Tf:\R^N\to X\) in the former case and \(Tf:\R^N\to\bddlin(X)\) in the latter. If \(\xi\in X\) and \(F:\R^N\to\C\) or \(F:\R^N\to\bddlin(X)\), define the function \(F\otimes\xi:\R^N\to X\) by \((F\otimes\xi)(x):=F(x)\xi\), where the last expression is the product of a scalar and a vector, or the action of an operator on a vector, respectively. With this notation, suppose that \(T\big(\varphi\otimes\xi\big)=(T\varphi)\otimes\xi\) for \(\varphi:\R^N\to\C\) and \(\xi\in X\). The adjoint \(T^*\) is defined via the duality \(\pair{g}{f}=\int \pair{g(x)}{f(x)}\ud\mu(x)\) between functions \(f:\R^N\to X\) and \(g:\R^N\to X^*\): for \(\varphi,\psi:\R^N\to\C\), \(\xi\in X\) and \(\xi^*\in X^*\),
\begin{equation*}
  \xi^*\big(\pair{\psi}{T\varphi}\xi\big)
  =\pair{\psi\otimes\xi^*}{T(\varphi\otimes\xi)}
  =:\pair{T^*(\psi\otimes\xi^*)}{\varphi\otimes\xi}
  =:\big(\pair{T^*\psi}{\varphi}\xi^*\big)(\xi),
\end{equation*}
and hence \(\pair{T^*\psi}{\varphi}=\big(\pair{\psi}{T\varphi}\big)^*\in\bddlin(X^*)\) for scalar-valued functions \(\varphi,\psi\).

Such a \(T\) is called a  \emph{Rademacher--Calder\'on--Zygmund operator} with kernel \(K\) if
\begin{equation}\label{eq:RTf}
  Tf(x)=\int_{\R^N}K(x,y)f(y)\ud\mu(y)
\end{equation}
for \(x\) outside the support of \(f\).
An operator \(T\) is said to satisfy the \emph{rectangular weak Rademacher-boundedness property} if there holds
\begin{equation*}
  \rbound\Big(\Big\{\frac{1}{\mu(R)}\int_{\R^N} 1_R\cdot T1_R \ud\mu:
      R\subset\R^N\text{ a rectangle}\Big\}\Big)\leq 1.
\end{equation*}
Recall that in \(Tb\) theorem~\ref{thm:operator}, this assumption is made for \(M_{b_2}TM_{b_1}\) in place of \(T\), where \(b_1,b_2\) are two fixed weakly accretive functions.
As in the scalar-kernel case, the simplifying assumption is made that \(T\) already defines an a priori bounded operator on \(L^p(\mu;X)\).

At this point, one can already explain the modifications in the proof of \(Tb\) theorem~\ref{thm:main}, except the part involving the paraproduct, which are required to get the operator-valued \(Tb\) theorem~\ref{thm:operator}. It is very simple: one just repeats the same proof, and the assumed Rademacher-boundedness conditions ensure that whenever one ``pulled out'' bounded scalar coefficients from the randomized series (which persist throughout the arguments), the same can be done with the operator coefficients by the very definition~\eqref{eq:defRbd}. This is by now completely standard in the study of operator-valued singular integrals (cf.~\cite{H:Tb,HW:T1,Weis}), and it would be redundant to say anything more here.

One still has to make sense of the actual \(Tb\) conditions and comment on their r\^ole in handling the paraproduct part of \(T\) in the operator-kernel case. To this end, observe first that formally
\begin{equation*}
\begin{split}
  \pair{T^* b_2}{b_1\varphi_Q}
  &=\big(\pair{b_2}{T(b_1\varphi_Q)}\big)^* \\
  &=\big(\pair{1_{2Q}b_2}{T(b_1\varphi_Q)}\big)^*
    +\big(\pair{1_{(2Q)^c}b_2}{T(b_1\varphi_Q)}\big)^*
  \in\bddlin(X^*).
\end{split}
\end{equation*}
The first term is the adjoint of the operator
\begin{equation*}
  \xi\in X\mapsto\pair{1_{2Q}b_2}{T(b_1\varphi_Q)}\xi
  =\pair{1_{2Q}b_2}{T(b_1\varphi_Q\otimes\xi)}\in X,
\end{equation*}
where the right side is well-defined, since \(f:=b_1\varphi_Q\otimes\xi\), and then \(Tf\), is in \(L^p(\mu;X)\) and \(1_{2Q}b_2\in L^{p'}(\mu)\). By \eqref{eq:RTf} and the fact that \(b_1\varphi_Q\) has a vanishing integral, the second term involves the pairing
\begin{equation*}
\begin{split}
  &\int_{(2Q)^c}b_2(x)\int_Q K(x,y)b_1(y)\varphi_Q(y)\ud\mu(y)\ud\mu(x) \\
  &=\int_{(2Q)^c}\int_Q b_2(x)[K(x,y)-K(x,y_Q)]b_1(y)\varphi_Q(y)\ud\mu(y)\ud\mu(x),
\end{split}
\end{equation*}
where \(y_Q\) is the centre of \(Q\) and the \(\bddlin(X)\)-valued double integral converges absolutely by~\eqref{eq:RCZK2} (even the uniform boundedness instead of Rademacher-boundedness would suffice here).

In \(Tb\) theorem~\ref{thm:operator}, it was assumed that
\begin{equation*}
  \Norm{T^*b_2}{\BMO^{p'}_{\lambda}(\mu;Z)}\leq 1,\qquad Z\subseteq\bddlin(X^*).
\end{equation*}
This condition can be interpreted as follows: There exists a function
\begin{equation*}
  h_2\in\BMO^{p'}_{\lambda}(\mu;Z)\subseteq\BMO^{p'}_{\lambda}(\mu;\bddlin(X^*))
\end{equation*}
of norm at most \(1\), such that
\begin{equation*}
  \pair{T^* b_2}{b_1\varphi_Q}=\pair{h_2}{b_1\varphi_Q}
  =\int h_2(x)b_1(x)\varphi_Q(x)\ud x
\end{equation*}
for all ``Haar'' functions \(\varphi_Q\). The space \(\BMO^{p'}_{\lambda}(\mu;Z)\) is defined just like the scalar-valued version (see~\eqref{eq:defBMO}), only using the norm of \(Z\) in place of the absolute value. Similarly one interprets the condition that \(Tb_1=h_1\in\BMO^{p}_{\lambda}(\mu;Y)\).

\begin{lemma}\label{lem:BMOumd}
For \(p\in(1,\infty)\) and \(h\in\BMO_{\lambda}^p(\mu;Z)\), where \(Z\) is a UMD space, there holds
\begin{equation*}
  \BNorm{\sum_{\ontop{Q\in\mathscr{D}^{\good};Q\subset R}{\ell(Q)\leq 2^{-r}\ell(R)}}
    \radem_Q\pair{h}{b_1\varphi_Q}\varphi_Q}{L^p(\prob\otimes\mu;Z)}
  \lesssim \mu(R)^{1/p}\Norm{h}{\BMO_{\lambda}^p(\mu;Z)}.
\end{equation*}
\end{lemma}

\begin{proof}
The proof is exactly the same as that of Lemma~\ref{lem:NTV51} dealing with \(Z=\C\); indeed, only the UMD property of \(\C\) was employed there.
\end{proof}

\begin{theorem}\label{thm:umdPara}
Under the hypotheses of $Tb$ theorem~\ref{thm:operator},
\begin{equation*}
  \Norm{\Pi_2 g}{\bddlin(L^{p'}(\mu;X^*))}\lesssim\Norm{T^*b_2}{\BMO^{p'}_{\lambda}(\mu;Z)}\leq 1.
\end{equation*}
\end{theorem}

\begin{proof}
This repeats the proof of Theorem~\ref{thm:paraproduct}. The initial considerations and the application of the tangent martingale trick of Theorem~\ref{thm:trick} for $X^*$-valued functions work in exactly the same way as before. In place of \eqref{eq:applyPf}, one applies the abstract paraproduct estimate of Theorem~\ref{thm:Pf} with $X_1=X^*$ and $X_3=L^q(\nu;X^*)$ ($q:=p'$) as before, but now with $X_2:=L^q(\nu;Z)\subset\bddlin(X_1,X_3)$, recalling that $Z\subset\bddlin(X^*)$. And the right side of \eqref{eq:applyPf} is then estimated as before, with Lemma~\ref{lem:BMOumd} in place of Lemma~\ref{lem:NTV51}. Finally, note that the last inequality of the assertion is here assumed as stated, contrary to the scalar-kernel Theorem~\ref{thm:paraproduct} where it was deduced from \(\Norm{T^*b_2}{\BMO^1_{\lambda}(\mu)}\leq 1\) and the other assumptions by the results of Nazarov, Treil and Volberg~\cite{NTV:Tb}.
\end{proof}

Thus all the arguments of $Tb$ theorem~\ref{thm:main} carry over to the setting of $Tb$ theorem~\ref{thm:operator}, and the sketch of the proof of the last-mentioned result is complete.

\section*{Acknowledgements}

Since the circulation of the first preprint version, the careful reading of the manuscript by Dr.~Antti V\"ah\"a\-kangas, and our discussions around this theme, led to several improvements throughout the work. I am particularly grateful for his discovery of a gap in an earlier version of the argument in Section~\ref{sec:contained}.


\bibliography{nonhomog}
\bibliographystyle{plain}

\end{document}